\documentclass[12pt]{article}
\usepackage{amssymb,latexsym, amsmath, amsthm, amsfonts, amssymb, enumerate, mathrsfs, paralist,cite,extpfeil}
\usepackage{cases}
\usepackage[title,titletoc]{appendix}
\usepackage{yfonts,eufrak}
\usepackage[usenames,dvipsnames,svgnames,x11names,table]{xcolor}
\usepackage{color}
\usepackage{oplotsymbl}
\usepackage{graphicx,overpic}
\usepackage{hyperref}
\hypersetup{
	CJKbookmarks=true
	colorlinks=true,
	linkcolor=blue,
	filecolor=blue,
	urlcolor=blue,
	citecolor=blue,
	bookmarksnumbered=true
}

\usepackage{geometry}
\def\rr{{\mathbb R}}

\def\nn{{\mathbb N}}

\def\az{\alpha}

\def\dz{\delta}
\def\Dz{\Delta}

\def\tz{\theta}
\def\vtz{\vartheta}

\def\gz{\gamma}
\def\Gz{\Gamma}

\def\A{\mathbf{A}}

\def\aa{\mathfrak{A}}
\def\B{\mathbf{E}}

\def\G{\mathbb G}
\def\H{\mathbb H}

\def\I{{\mathbb{I}}}
\def\N{\mathbb N}

\def\PP{\mathbb{P}}
\def\R{\mathbb{R}}
\def\RR{\mathrm{R}}

\def\T{{\mathrm{T}}}

\def\UBB{\mathbb U}

\def\X{{\mathrm{X}}}

\def\Y{{\mathrm{Y}}}

\def\ep{\epsilon}

\def\mrr{\mathrm{R}}
\def\goodf{C_c^\infty(\G)}
\DeclareMathOperator\dif{d\!}
\def\pik{{\bf PI}}
\def\transpose{\mathsf{T}}

\def\ss{\mathcal{S}}

\def\ntt{N_{3,2}}

\def\m{\mathfrak{m}}

\def\lsim{\lesssim}
\def\gsim{\gtrsim}

\def\jac{\mathrm{Jac}\,}

\def\vol{\mathrm{vol}}
\def\U{\mathbb{U}}

\newtheorem{theo}{Theorem}[section]

\newtheorem{lem}[theo]{Lemma}

\newtheorem{prop}[theo]{Proposition}

\newtheorem{ass}[theo]{Assumption}

\title{On gradient estimates of the heat semigroups on step-two Carnot groups}
\author{%Hong-Quan Li, 
Sheng-Chen Mao, Ye Zhang}
\date{}

\begin{document}
	
	\renewcommand{\theequation}{\thesection.\arabic{equation}}
	\setcounter{equation}{0} \maketitle
	
	\vspace{-1.0cm}
	
	\bigskip
	
	%\bigskip{\color{blue}\underline{{\it Updated on \today}}}\bigskip
	
	{\bf Abstract.} In this work, we give a sufficient condition for a step-two Carnot group to satisfy the quasi Bakry--\'Emery curvature condition. As an application, we establish the gradient estimate for the heat semigroup on the free step-two Carnot group with three generators $N_{3,2}$. Moreover, higher order gradient estimates and Riemannian counterparts are also deduced under an extra condition.

	\medskip
	
	
	\medskip
	
	
	\medskip
	
%	\addctmk{Contents}\tableofcontents
	
%	\medskip
	
\section{Introduction}
\setcounter{equation}{0}

In the groundbreaking works of Sturm \cite{S060,S06} and Lott--Villani \cite{LV09}, using optimal transport, the notion of curvature-dimension condition is generalized to non-smooth metric measure spaces. Furthermore, it is stable under measured Gromov--Hausdorff convergence and implies almost all the geometric and functional analytic estimates from Riemannian geometry, even in the non-smooth setting. 

However, this condition fails to hold on sub-Riemannian manifolds. The first counterexample is given by Juillet \cite{J09} on the Heisenberg group, one of the simplest examples of sub-Riemannian manifolds. Furthermore, it is shown recently in \cite{J21} that sub-Riemannian manifolds cannot admit any curvature-dimension condition. See also \cite{RS23, MR23}.

Consequently, several weak notions of the curvature-dimension condition are introduced to this sub-Riemannian setting, including quasi Bakry--\'Emery curvature condition \cite{Li06, BBBC08, E10, HL10, Z21, Z23}, measure contraction property \cite{J09, R13, R16, BR18, BR20, GZ24}, sub-Riemannian interpolation inequalities \cite{BKS18, BKS19, BR19, BMR24}, other generalized curvature-dimension condition \cite{GT16, GT162, BG17, M21, Q15}.

In this paper we focus on quasi Bakry--\'Emery curvature condition in the setting of step-two Carnot groups.

\subsection{Quasi Bakry--\'Emery curvature condition}

Let $\G$ be a step-two Carnot group. We use $\nabla$ and $\Delta$ to denote the horizontal gradient and (canonical) sub-Laplacian respectively (see Section \ref{s2} for the precise definitions). The quasi Bakry--\'Emery curvature condition is the following inequality for the heat semigroup $(e^{h \Delta})_{h > 0}$: 
\begin{equation} \label{hsge}
\left|\nabla e^{h \Delta} f\right|(g) \leq C \, e^{h \Delta}(|\nabla f|)(g), \quad \forall \, g \in\G,  h>0, f \in C_c^{\infty}(\G)
\end{equation}
for some constant $C > 1$. 

We remark that the constant $C$ cannot be equal to or smaller than $1$ since otherwise it would imply the curvature-dimension condition $\mathrm{CD}(0,\infty)$. Due to this fact, it is called quasi Bakry--\'Emery curvature condition in \cite[Definition 2.15]{A20}. 

By H\"older's inequality, \eqref{hsge} implies the following weak version:
\begin{align} \label{grap}
|\nabla e^{h \Delta} f|^p(g) \le C_p \, e^{h \Delta} (|\nabla f|^p)(g),\quad \forall \, g \in \G, h > 0, f \in C_c^\infty(\G) 
\end{align}
for $p > 1$. 

The weak inequality \eqref{grap} is established in \cite{DM05} for the first Heisenberg group $\H^1$ by using the  Malliavin Calculus, and later in \cite{M08, GT19} for general Carnot groups. 

For the strong inequality \eqref{hsge}, it is first established in the paper \cite{Li06}. Then more proofs are given in \cite{BBBC08}.  Generalizations to the H-type groups and nonisotropic Heisenberg groups can be found in \cite{E10, HL10, Z21}.  It can also be obtained via the method of $U$-bounds in \cite{HZ10, IKZ11}.

However, all the proofs in the above papers depend heavily on the precise upper and lower bound estimates of the heat kernel and upper bound of the gradient of the heat kernel, which are usually derived from the uniform heat kernel asymptotics thanks to the explicit formula of the heat kernel given by oscillatory integral (see for example \cite{Li07, E09, Li10, LZ19}), as well as on the computation of the Jacobian determinant of the sub-Riemannian exponential map. Despite some attempts such as the one in \cite{LZ212}, owing to the complexity of the geometry of the underlying group, it seems still challenging to deduce the precise upper and lower bound estimates of the heat kernel in general groups. Furthermore, recently the uniform heat kernel asymptotic, as well as the precise upper and lower bound estimates for the heat kernel, is established on free step-two Carnot group with three generators $N_{3,2}$ in \cite{LMZ23}, but it appears that the expression of the bound is too complicated to apply. As for the computation of the sub-Riemannian exponential map, it is complicated even on the free step-two Carnot group with three generators $N_{3,2}$ and should be more difficult in general groups since the key information such as the cut times, cut loci are still unknown. 

Recently some stability properties of \eqref{hsge} under tensorization and projection are studied in \cite{Z23, BGS24} to explore more groups supporting \eqref{hsge}. See also \cite{GL22} for basic ideas of this approach.

In the present paper, we provide a sufficient condition for \eqref{hsge} on a general step-two Carnot group to overcome these problems, whose handiness is illustrated by several concrete examples in Section \ref{s4}. In other words, by abstracting the main ingredient of the proof given in \cite{Li06, HL10}, we find that the key difficulty is an inequality which is closely related to the measure contraction property w.r.t. a special measure. In this way, one can avoid using the precise heat kernel bounds and computing the Jacobian determinant if such inequality is verified.  Although in our examples we still have to fall back on the precise heat kernel bounds, one can actually get rid of the computation of the Jacobian determinant of the sub-Riemannian exponential map via the small time asymptotics of the heat kernel (see Subsection \ref{s42}) and we believe that this argument will shed new light on the proof of this inequality. See the examples in Section \ref{s4} for more details.
 
All the pertinent assumptions are listed in the following subsection.

\subsection{Assumptions and measure contraction property}

In the rest of this paper we use $(p_h)_{h > 0}$ to denote the convolution kernel of the heat semigroup $(e^{h \Delta})_{h > 0}$, $d$ the Carnot--Carath\'eodory distance (sub-Riemannian distance), $o$ the identity element, and $\vol$ the Haar measure of $\G$. Furthermore, for  $g\in\G$, $E \subset \G$ and $s \in [0,1]$, we define the set of $s$-intermediate points as
\begin{equation}\label{sinter}
Z_s(g,E) := \left\{ g' \in \G; \,  \exists \, g'' \in E \ s.t. \ d(g,g') = s\, d(g,g''), d(g'',g') = (1-s) \, d(g,g'') \right\}.
\end{equation}

Let $d(g) := d(g,o)$, $p := p_{1}$, $\dif \mu := e^{\frac{d^2}{4}} p \, \dif \vol$ and
\[
\mathscr{B}:=\{E\subset\mathbb{G}; E \,\,\mbox{is measurable, bounded and}\,\, 0<\mathrm{vol}(E)<\infty\}.
\]
Now we introduce the assumption imposed on the group $\G$:
\begin{ass}\label{ass}
Let $\G$ be a step-two Carnot group. There exist constants $N > 0$ and $C > 0$ such that
\begin{equation} \label{assmcp}
\mu(Z_s(o,E)) \ge C^{-1} s^{N} \mu(E), \quad \forall \, s \in [0,1],\, E \in\mathscr{B}. 
\end{equation}
\end{ass}

We remark that \eqref{assmcp} can be weakened to hold only on $[\frac12,1]$ as shown in Subsubsection \ref{coresec}, and it indeed can be verified for concrete examples cited in Section \ref{s4}. Moreover, the inequality \eqref{assmcp} is actually an assumption of measure contraction property type.

Recall that the measure contraction property is introduced in \cite{S06} or \cite{O07}. In the setting of step-two Carnot groups, the measure contraction property $\mathrm{MCP}(0,N)$ can be written as 
\begin{align}\label{defmcp}
\vol(Z_s(g,E)) \ge  s^{N} \vol(E), \quad \forall \, g \in \G, s \in [0,1], E \in\mathscr{B}.
\end{align}
Here we have restricted the set $E$ to being bounded in \eqref{defmcp}, which is equivalent to the case without that restriction as in \cite[Definition 2.1]{O07} through a limiting argument. From the left invariance, it suffices to consider \eqref{defmcp} for $g = o$. As a result, our assumption \eqref{assmcp} essentially differs from \eqref{defmcp} up to a constant $C^{-1}$ and a change of measure. However, since $\mu$ is not left invariant,  if we rewrite assumption \eqref{assmcp} into the general case, then it should be 
\[
\mu_g(Z_s(g,E)) \ge C^{-1} s^{N} \mu_g(E), \quad \forall \, g \in \G, s \in [0,1],\, E \in\mathscr{B},
\]
where $\dif \mu_g := p(g,\cdot) e^{\frac{d(g,\cdot)^2}{4}} \dif \vol$ with $p(g,g') := p(g^{-1} \cdot g')$. Thus there is some slight difference comparing to the usual change of measure for measure contraction property. Recall also that in the paper of Juillet \cite{J09}, the measure contraction property $\mathrm{MCP}(0,2n + 3)$ is established on Heisenberg groups $\H^n$ and thus the measure contraction property becomes one of the weak notions of the curvature-dimension condition on sub-Riemannian manifolds. Actually this property holds on general step-two Carnot group.
\begin{theo}[Theorem 1.4 of \cite{BR20}]\label{tmcp}
Given a step-two Carnot group $\G$, there exists a constant $N > 0$ such that the measure contraction property $\mathrm{MCP}(0,N)$ holds.
\end{theo}

By employing the sub-Riemannian exponential map, we can obtain the equivalent statement for such kind of inequality; see Proposition \ref{HCP} below.

For more related results, we refer to \cite{BR20, R13, R16, BR18, GZ24, BT23, BMRT24}. However, in view of the small time result of the heat kernel below (Lemma \ref{sth}), roughly speaking, the measure $\mu$ also affects the number $N$ and for this reason, our assumption \eqref{assmcp} cannot be derived directly from the measure contraction property $\mathrm{MCP}(0,N)$. See also the examples in Section \ref{s4}.

\medskip

For estimates of higher order derivatives, we need an extra assumption as follows.

\begin{ass}\label{ass2}
Let $\G$ be a step-two Carnot group. For every $k=2,3,\ldots$ there exists a constant $C_k > 0$ such that
\begin{align}\label{hka2}
|\nabla^k p(g)| &\le C_k \, (1 + d(g))^k \, p(g), \qquad \forall \, g \in \G.
\end{align}
\end{ass}

Here we recall that the inequality \eqref{hka2} for $k = 1$ is known on step-two Carnot groups. To be more precise, it follows from \cite[Proposition 2.21]{BG17} that step-two Carnot group satisfies the generalized curvature-dimension inequality $\mathrm{CD}(0,\rho,\kappa,q)$ in the sense of Baudoin--Garofalo with $\rho,\kappa > 0$. Then \cite[Theorem 3.1]{Q15} gives the following gradient estimate of the heat kernel. Note that although in \cite[Theorem 3.1]{Q15} it is assumed that the invariant measure is finite, it does not play any role in the proof so one can apply it to step-two Carnot groups.  See Appendix \ref{app} for more details.
%We have already confirmed this with the author.
%This has been confirmed with the author.}
\begin{lem} \label{geh}
	Let $\G$ be a step-two Carnot group. Then there is a constant $C>0$ such that
	\begin{equation}\label{assud}
		|\nabla p(g)| \le  C (1+d(g)) \, p(g), \qquad \forall \, g \in \G.
	\end{equation}
\end{lem}

Now we can formulate our main results of this work.

\subsection{Main results}

\begin{theo} \label{thm1}
Let $\G$ be a step-two Carnot group. If Assumption \ref{ass} holds, then the  quasi Bakry--\'Emery curvature condition \eqref{hsge} is valid for some constant $C > 1$ (not necessarily the same $C$ in Assumption \ref{ass}).
\end{theo}

For higher order derivatives, we have:

\begin{theo} \label{thm2}
Let $\G$ be a step-two Carnot group. If Assumptions \ref{ass} and \ref{ass2} hold, then for any $k\in\nn^*$ there is a constant $C(k)>0$ such that
\begin{equation} \label{hsg2}
	\left|\nabla^k e^{h \Delta} f\right|(g) \leq C(k)\, e^{h \Delta}(|\nabla^k f|)(g), \quad \forall \, g \in \G, h>0, f \in C_c^{\infty}(\G).
\end{equation}
\end{theo}

As in \cite{Z23},  the results can also be generalized to the Riemannian gradient $\nabla_\RR$ and full-Laplacian $\Delta_\RR$ (see Section \ref{s2} for the precise definition). 

\begin{theo} \label{thm3}
Let $\G$ be a step-two Carnot group. 
\begin{enumerate}[(i)]
\item If Assumption \ref{ass} holds, then there is a constant $C > 1$ (not necessarily the same $C$ in Assumption \ref{ass}) such that
	\begin{equation*}
		\left|\nabla_\RR e^{h \Delta_\RR} f\right|(g) \leq C\, e^{h \Delta_\RR}(|\nabla_\RR f|)(g), \quad \forall \, g \in \G, h>0, f \in C_c^{\infty}(\G).
	\end{equation*}
\item If additionally Assumption \ref{ass2} holds, then	for any $k\in\nn^*$  there is a constant $C(k)>0$ such that
	\begin{equation*}
		\left|\nabla_\RR^k e^{h \Delta_\RR} f\right|(g) \leq C(k)\, e^{h \Delta_\RR}(|\nabla_\RR^k f|)(g), \quad \forall \, g \in \G, h>0, f \in C_c^{\infty}(\G).
	\end{equation*}

\end{enumerate}
\end{theo}

Our main example is the free step-two Carnot group with three generators $N_{3,2}$, where the uniform heat kernel asymptotic, as well as the precise upper and lower bound estimates of the heat kernel, has been established recently in \cite{LMZ23}. We will show in Subsection \ref{s42} that $N_{3,2}$ satisfies Assumptions \ref{ass} and \ref{ass2}. This solves the open problem in \cite{Q11}.

\subsection{Notation} 
We denote the set of nonnegative integers by $\nn$ and write $\nn^*:=\nn\setminus\{0\}$. For $a\in\rr$ we use $\lfloor a\rfloor$ to represent the largest integer not exceeding $a$. Given $k\in\nn^*$ and $\az=(\az_1,\ldots,\az_k)\in\nn^k$, we set $|\az|:=\sum_{i=1}^k\az_i$. 
Supposing $u\in\rr^3$, we will adopt the following convention
\begin{equation*}
\hat{u}:= \begin{cases}\frac{u}{|u|}, & \text { if } u \in \mathbb{R}^{3} \backslash\{0\},
\\ 0, & \text { if } u=0.\end{cases}
\end{equation*}

Given a $C^1$ map $F$, we denote its Jacobian determinant by $\jac(F)$. Furthermore, given a $C^2$ function $G$, we denote the Hessian matrix by $\mathrm{Hess}(G)$.

Throughout, we will use the symbols $C$ and $c$  to denote implicit positive constants which may change from one line to the next. And if necessary we will specify with a subscript on which parameters the values of $C$ and $c$ depend.

Let $w$ be a non-negative function.  By $f = O(w)$ (resp. $f = o(w)$) we mean $|f| \leq C w$ (resp. $f/w \to 0$ under some limit process). If $f$ is also real-valued, we say $f \lesssim w$ (resp. $f \gtrsim w$) if $f \leq C w$ (resp. $f \ge C\, w$). We also adopt the notation $f \sim w$ if $f \lesssim w$ and $w \lesssim f$. When the implicit constant depends on the parameter $\az_0$, we will write $O_{\az_0},o_{\az_0},\lesssim_{\az_0}, \sim_{\az_0}$, etc. 

Lastly, in this work all vectors will be regarded as column vectors unless otherwise stated. However, a column vector $t$ in $\R^q$ with scalar coordinates $t_1, \ldots, t_q$ will be written simply as $(t_1, \ldots, t_q)$. The matrix transpose is denoted by $\cdot^\transpose$.

\subsection{Structure of the paper}

We recall basic properties of step-two Carnot groups and known results in   Section \ref{s2}. In Section \ref{s31} we give the proof of Theorem \ref{thm1} while the proofs of Theorems \ref{thm2} and \ref{thm3} are provided in Section \ref{s32}.  Finally in Section \ref{s4} we will furnish concrete examples to which our main theorems can be applied.

\section{Preliminaries}\label{s2}
\setcounter{equation}{0}

\subsection{Step-two Carnot groups}

Recall that a connected and simply connected Lie group $\G$ is a step-two Carnot group
if its left-invariant Lie algebra $\mathfrak{g}$ admits a stratification
\begin{align*}
\mathfrak{g} = \mathfrak{g}_1 \oplus \mathfrak{g}_2, \quad
[\mathfrak{g}_1, \mathfrak{g}_1] = \mathfrak{g}_2, \quad
[\mathfrak{g}_1, \mathfrak{g}_2] = \{0\},
\end{align*}
where $[\cdot,\cdot]$ denotes the Lie bracket on $\mathfrak{g}$. We identify $\G$ and $\mathfrak{g}$ via the exponential map. As a result, $\G$ can be considered as $\R^q \times \R^m$, $q, m \in \N^* = \{1, 2, 3, \ldots\}$  with the group law
\begin{align}\label{gst}
(x , t) \cdot (x^{\prime}, t^{\prime}) =
\left(x + x^{\prime}, t + t^{\prime} + \frac{1}{2}\langle  \U x, x^{\prime} \rangle \right), \quad g := (x, t) \in \R^q \times \R^m,
\end{align}
where
\begin{align*}
\langle\U x,  x^{\prime} \rangle := (\langle U^{(1)} x, x^{\prime} \rangle, \ldots, \langle   U^{(m)} x,x^{\prime} \rangle) \in \R^m.
\end{align*}
Here $\U = \{U^{(1)},\ldots,U^{(m)}\}$ is an $m$-tuple of linearly independent $q \times q$ skew-symmetric matrices with real entries and $\langle \cdot, \cdot \rangle$ (or $\mbox{} \cdot \mbox{}$ in the sequel when there is no ambiguity) denotes the usual inner product on $\R^q$. The Haar measure $\vol$ on $\G$ is the Lebesgue measure. One can refer to \cite{BLU07} for more details. Let $U^{(j)} = (U^{(j)}_{l, k})_{1 \leq l, k \leq q}$ ($1 \leq j \leq m$). The canonical bases of $\mathfrak{g}_1$ and $\mathfrak{g}_2$ are defined by the left-invariant vector fields on $\G$:
\begin{align*}
\X_l(g) : = \frac{\partial}{\partial x_l} + \frac{1}{2} \sum_{j = 1}^m \Big( \sum_{k = 1}^{q} U^{(j)}_{l, k} x_k \Big) \frac{\partial}{\partial t_j}, \quad 1 \leq l \leq q, \quad \mbox{and} \quad \T_\ell(g) := \frac{\partial}{\partial t_\ell}, \quad 1  \leq \ell \leq m
\end{align*}
respectively. The horizontal gradient and canonical sub-Laplacian are defined repectively by
\[
\nabla  := (\X_1, \ldots, \X_q), \qquad  \Delta := \sum\limits_{l = 1}^q \X_l^2.
\]
Correspondingly, the Riemannian gradient and full-Laplacian are given by
\[
\nabla_\mrr := (\X_1, \ldots, \X_q,\T_1, \ldots, \T_m), \qquad  \Delta_\mrr := \Delta + \Delta_\T \quad \mbox{with} \quad  \Delta_\T := \sum_{\ell =  1}^{m} \T_\ell^2.
\]

Now for $k\in\nn^*$, we write $\X^I:=\X_{i_1}\cdots \X_{i_k}$ for $I=(i_1,\ldots,i_k) \in (\N^*)^k$ with $1 \le i_1,\dots,i_k \le q$ (the letter $I$ will be reserved for such tuples in what follows), and set $|I|' := k$. Then we define
\begin{equation*}
	|\nabla^k f|:=\left[\sum_{I: |I|'=k} |\X^I f |^2 \right]^\frac12
\end{equation*}
with the sum taken over all tuples $I$ of length $k$. Similar notation is also used for $|\nabla_\RR^k f|$.

Let $(e^{h \Delta})_{h > 0}$ be the heat semigroup and $(p_h)_{h > 0}$ the
heat kernel , i.e. the fundamental solution of $\frac{\partial}{\partial h} - \Delta$. It is well-known that $0 < p_h \in C^{\infty}(\R^+ \times \G)$, and (see for example \cite{VSC92})
\[
e^{h \, \Delta} f(g) = f * p_h(g) = \int_{\G} f(g_*) \, p_h(g_*^{-1} \cdot g) \, \dif g_*.
\]

The dilation on $\G$ is defined by
\begin{align} \label{nDS}
\delta_r(g) = \delta_r(x, t) := (r \, x, r^2 \, t), \quad \forall \, r > 0, \ g = (x, t) \in \G.
\end{align}
A function $f$ on $\G$ is called ($\delta_r$-)homogeneous of order $a\in\rr$ if 
\begin{equation*}
    f(\dz_r(g)) = r^a f(g), \quad \forall\,r>0,\,g\in\G.
\end{equation*}
The ($\delta_r$-)homogeneous order of a polynomial function $P(x,t)=\sum_{\az_1,\az_2} c_{\az_1,\az_2}\, x^{\az_1}t^{\az_2}$ on $\G$ is the integer $\max\{|\az_1|+2|\az_2|; c_{\az_1,\az_2}\neq0\}$.

It is also well-known that the heat kernel $(p_h)_{h > 0}$ admits the following properties (cf. for example \cite[Chapter~1~G]{FS82}  and references therein):
\begin{align}\label{ehk0}
p_h(g) = \frac{1}{h^{\frac{Q}{2}}} \, p_1\left(\delta_{\frac{1}{\sqrt{h}}}(g)\right), \qquad p_h(g) = p_h(g^{-1}),
\end{align}
where the homogeneous dimension $Q := q + 2m$. 
Recall that we have put $p := p_1$. 

\subsection{Left-invariant sub-Riemannian structure on $\G$}

For more details of the sub-Riemannian structure, we refer the reader to \cite{M02, R14, ABB20} and references therein for further details.

The group $\G$ is endowed with
the sub-Riemannian structure, induced by a scalar product on $\mathfrak{g}_1$, with respect to which $\{\X_l\}_{1 \le l \le q}$ are orthonormal (and the norm induced by this scalar product is denoted by $\| \cdot \|$).

A horizontal curve $\gamma: [0, \ 1] \to \G$ is an absolutely continuous path such that
\[
\dot{\gamma}(s) = \sum_{j = 1}^q u_j(s) \X_j(\gamma(s)) \qquad \mbox{for a.e. } s \in [0, \ 1],
\]
and we define its length as follows
\begin{align*}
\ell(\gamma) := \int_0^1 \|\dot{\gamma}(s)\| \, \dif s = \int_0^1 \sqrt{\sum_{j = 1}^q |u_j(s)|^2} \, \dif s.
\end{align*}
The Carnot--Carath\'eodory (or sub-Riemannian) distance between $g, g' \in \G$ is then
\[
d(g, g') := \inf\left\{\ell(\gamma); \ \gamma(0) = g, \ \gamma(1) = g', \gamma \mbox{ horizontal} \right\}.
\]
From definition it is easy to see that the Carnot--Carath\'eodory is left-invariant, namely
\begin{align}\label{lid}
d(g_* \cdot g, g_* \cdot g') = d(g,g'), \qquad \forall \, g_*, g, g' \in \G. 
\end{align}

Now let $d(g) = d(x,t) := d(g,o) = d((x,t),(0,0))$ denote the Carnot--Carath\'eodory (or sub-Riemannian) distance between the identity element $o = (0,0)$ and $g  = (x,t)\in \G$. Then we have
the following invariant property under inversion by \eqref{lid}:
\begin{align}\label{lid2}
d(g) = d(g,o) = d(o,g^{-1}) = d(g^{-1}, o) = d(g^{-1}),
\end{align}
and the scaling property:
\begin{align} \label{scap}
d(\delta_r(g)) = r \, d(g), \quad \forall \, r > 0, \ g \in \G.
\end{align}
Moreover,  the following equivalence of the Carnot--Carath\'eodory distance with a homogeneous norm
(see for example \cite[Proposition 5.1.4]{BLU07})  will be useful in the proof:
\begin{align} \label{ehd}
d(x, t)^2 \sim |x|^2 + |t|, \qquad \forall \, (x, t) \in \G.
\end{align}

We use the notation $B(g, r) := \{g'; \, d(g,g')<r\}$ ($g \in \G$, $r > 0$) to denote the open (Carnot--Carath\'eodory) ball centered at $g$ with radius $r$.  From \eqref{scap} it is easy to show  that 
\begin{align}\label{volhomo}
\vol(B(g,r))={\rm\bf C_\G} \,r^Q \qquad \mbox{with} \qquad {\rm\bf C_\G} :=\vol(B(o,1)).
\end{align}

\medskip

In the setting of step-two Carnot groups,
it is well-known that all shortest geodesics are projections of normal Pontryagin extremals, that is, the integral curves of the sub-Riemannian Hamiltonian in $T^* \G$. See for example \cite[\S~20.5]{AS04} or \cite[Theorem 2.22]{R14}. As a result, after identifying $T^*_o \G$ with $\R^q \times \R^m$, we can introduce the sub-Riemannian exponential map based at $o$, $\exp: \, \R^q \times \R^m \to \G$, in such a way that for every $(\zeta, \tau) \in \R^q \times \R^m \cong T^*_o \G$, the curve $s \mapsto \exp\{s(\zeta, \tau)\}$ is the projection of normal Pontryagin extremals onto $\G$. Here it should be pointed out that although we use the notation ``$\exp$'', the sub-Riemannian exponential map is different from the Lie group exponential map we used above.

For the sake of future usage, here we recall the formula of the sub-Riemannian exponential map $\exp$ on step-two Carnot groups. For more details, we refer to \cite[\S~13.1]{ABB20}.

Recall that $\langle\U x,  x^{\prime} \rangle := (\langle U^{(1)} x, x^{\prime} \rangle, \ldots, \langle   U^{(m)} x,x^{\prime} \rangle) \in \R^m$ and $\U = \{U^{(1)},\ldots,U^{(m)}\}$. For $\tau \in \R^m$ we define
\[
\widetilde{U}(\tau) := \sum_{j = 1}^m \tau_j \, U^{(j)}.
\]
With this notation, we have
\begin{align}\label{Expexp}
\exp(\zeta, \tau) = (x(1),t(1)),
\end{align}
with
\begin{align*}
\zeta(s) := e^{ s \, \widetilde{U}(\tau)} \, \zeta, \quad
x(s) := \int_0^s \zeta(r) \, \dif r, \quad
t(s) := \frac{1}{2} \int_0^s \langle \UBB \, x(r), \zeta(r) \rangle \, \dif r.
\end{align*}
From the formula above we can check that if $\exp(\zeta,\tau) = (x,t)$, then % we have
\begin{gather} \label{symExp}
\exp(- e^{\widetilde{U}(\tau)} \, \zeta, - \tau) = (-x , -t) = -(x,t).
\end{gather}

\medskip

In this paper, the cut locus of $o$, $\mathrm{Cut}_o$, is defined as
\begin{align} \label{DCUT}
\mathrm{Cut}_o := \mathcal{S}^c, \quad \mbox{ with } \, \mathcal{S} := \{g; \, \mbox{$d^2$ is $C^{\infty}$ in a neighborhood of $g$}\}.
\end{align}
By definition $\mathrm{Cut}_o$ is closed. Furthermore, it has measure zero (cf. \cite[Proposition 15]{R13}). Let 
\begin{align*}
\mathcal{D} := \{(\zeta, \tau) &\in \R^q \times \R^m \cong T^*_o \G; \, s \mapsto \exp\{s(\zeta, \tau)\} (0 \le s \le 1) \\
&\mbox{ is the shortest geodesic joining $o$ to $\exp(\zeta, \tau) \in \mathcal{S}$}\}.
\end{align*}
From the discussion of \cite[\S~2.2]{LZ21}, we can conclude the following proposition.

\begin{prop} \label{propgeo}
The sub-Riemannian exponential map $\exp$ based at $o$ is a diffeomorphism from $\mathcal{D}$ to $\mathcal{S}$. Furthermore, for $(\zeta, \tau) \in \mathcal{D}$, the curve $s \mapsto \exp\{s(\zeta, \tau)\} (0 \le s \le 1)$ is the shortest geodesic joining $o$ to $\exp(\zeta, \tau)$. In particular, we have $s (\zeta, \tau) \in \mathcal{D}$ for every $s \in (0,1]$ and $d(\exp(\zeta, \tau)) = |\zeta|$.
\end{prop}

\subsection{Some known results on step-two Carnot groups}

In this section we collect some results on step-two Carnot groups which are useful in the proof of main theorems.

\medskip

Let us first recall the known pointwise estimates of the heat kernel (see, e.g., \cite{DP89, C93,S96,S04, CS08} for the upper bounds and \cite{V90} for the lower bounds). %However, it turns out that we only need the lower bound \eqref{h_down} in \eqref{low_use} by taking suitable $\dz$. Moreover, the upper bound \eqref{h_up} is not necessary and we only need the fact that $p(\cdot)$ is a Schwartz function instead. 
\begin{lem} \label{h_es1}
	Let $\G$ be a step-two Carnot group. Then there is a constant $C>0$ such that
	\begin{equation}\label{h_up}
		p(g) \le C (1+d(g))^{Q-1} e^{-\frac{d(g)^2}{4}},\quad \forall\, g\in\G.
	\end{equation}
Moreover, for any $\dz\in(0,1)$ there is constant $C(\dz)$ depending only on $\dz$ and $\G$ such that
\begin{equation}\label{h_down}
	p(g) \ge C(\dz)\, e^{-\frac{d(g)^2}{4(1-\dz)}},\quad \forall\, g\in\G.
\end{equation}
\end{lem}

The following representation formula is a special case of \cite[Theorem 1]{LW98}, together with \eqref{volhomo}. 
\begin{lem} \label{lem1} 
Let $\G$ be a step-two Carnot group. Then there exists a constant $C>0$ such that, for any ball $B \subset \G$, $f\in\goodf$, and $g\in B$,
	\begin{equation*}
		\left|f(g)-f_{B}\right| \leq C \int_{B} \frac{\left|\nabla f(g')\right|}{d(g', g)^{Q-1}} \dif g'
	\end{equation*}
Here and subsequently, $f_{B}$ represents the average of $f$ over $B$, that is,
 \begin{equation*}
	f_{B}:= \frac{1}{\vol(B)}\int_B f(g)\dif g.
\end{equation*}
\end{lem}

\medskip

We also need the higher order Poincar\'e inequality which can be obtained by iteration of the $1$-Poincar\'e inequality (i.e. the case $k = 1$ below, cf. \cite[Proposition 11.17]{HK00}). The existence of the $P_k(B,f)$ is guaranteed by \cite[Theorem F]{L00}. See also \cite[p. 168]{LLW02}.

\begin{lem} \label{lem2}
Let $\G$ be a step-two Carnot group and $k\in\nn^*$. Then there is a constant $C_k>0$ such that for any ball $B$ with radius $r(B) > 0$ and $f\in \goodf$,
	\begin{equation} \label{PIk}
		\int_B |f(g) - P_k(B,f)(g)| \dif g \le C_k \, r(B)^k \int_B |\nabla^k f(g)| \dif g.
	\end{equation}
Here $P_k(B,f)$ denotes the polynomial of homogeneous order less than $k$ which satisfies
\begin{equation*}
	\int_B \X^I (f-P_k(B,f)) (g)\dif g=0,
\end{equation*}
for all $k$-tuples $I$ obeying $|I|'<k$. Especially, when $k=1$ we easily see that $P_1(B,f)=f_{B}$, and \eqref{PIk} is just the usual $1$-Poincar\'e inequality.
\end{lem}

The inequality \eqref{PIk} will be referred to as $\pik(k)$.

\section{Proof of Theorem \ref{thm1}}\label{s31}
\setcounter{equation}{0}

Suppose that Assumption \ref{ass} holds. Using the translation  and the dilation properties, we only have to show \eqref{hsge} for $h=1$ and $g=o$. Moreover, by the gradient estimate \eqref{assud}, together with \eqref{ehk0} and \eqref{lid2}, we have 

\begin{equation*}
    |\nabla p(g^{-1})| \lsim (1+d(g^{-1}))\, p(g^{-1}) = (1+d(g))\, p(g),
\end{equation*}
which, together with the stochastic completeness (that is, $\int p = 1$) of the heat kernel implies that
\begin{align*}
    |\nabla e^{\Delta}(f)(o)| &= \left|\int_\G \nabla p(g^{-1}) (f(g)-f_{B(o,1)}) \dif g \right| \lsim \int_\G \left| f(g)-f_{B(o,1)} \right| (1+d(g)) \, p(g) \dif g \\
    &=\int_{B(o, \A)} + \int_{B(o, \A)^c} =: J_1 +J_2,
\end{align*}
where $\A \gg 1$ is a large constant to be chosen later. As a result, it suffices to establish the following estimates
\begin{equation}\label{hsg1}
    J_i \lsim \int_\G |\nabla f(g)|\, p(g) \dif g, \quad i=1,2. 
\end{equation}

The proof for $i=1$ of \eqref{hsg1} is simple. For $i=2$ we proceed by adapting some ideas from \cite{Li06,HL10}. More precisely, we split the integrand into finite small pieces along the shortest geodesic from $o$ to an integration variable, and on each small piece the representation formula and the Poincar\'e inequality are applied as in \cite{Li06,HL10}. The main differences are: 
	\begin{compactenum}[(i)]
		\item Delicate upper and lower bound estimates of the heat kernel are not required.  We only need rough bounds \eqref{h_up} and \eqref{h_down}.
  %We only need a rough lower bound \eqref{h_down};
		\item Assumption \ref{ass} offers us an approach to  bypass the technicalities in calculating the Jacobian determinant of the sub-Riemannian exponential map, which seems quite complex, even on H-type groups; see \cite[\S~4]{HL10}. On the other hand, one can verify Assumption \ref{ass} through other  means (see, e.g., Subsection \ref{s42} below).
	\end{compactenum}

\subsection{Estimation of $J_1$}\label{sJ1}
The positivity of the heat kernel implies that, if $D>0$, then
\begin{equation}\label{posi_p}
	\mbox{$p(\cdot) \sim_D 1$ \quad on \quad $B(o,D)$.}
\end{equation}
As a result, we have
\begin{equation*}
    J_1 \lsim \int_{B(o,\A)} \left| f(g)-f_{B(o,1)} \right| \dif g \le \int_{B(o,\A)} \left| f(g)-f_{B(o,\A)} \right| \dif g +\int_{B(o,\A)} \left| f_{B(o,1)}-f_{B(o,\A)} \right| \dif g.
\end{equation*}
By the $1$-Poincar\'e inequality \pik(1), the first item on the rightmost side is majorized by 
$$C \A \int_{B(o,\A)}|\nabla f(g)| \dif g.$$
While the second one equals
\begin{align*}
   \vol(B(o,\A)) \left| f_{B(o,1)}- f_{B(o,\A)} \right| &\le \frac{\vol(B(o,\A))}{\vol(B(o,1))}\int_{B(o,1)} \left| f(g') - f_{B(o,\A)} \right| \dif g' \\
   &\lsim \, \int_{B(o,\A)} \left| f(g') - f_{B(o,\A)} \right|  \dif g' \lsim \int_{B(o,\A)} |\nabla f(g')| \dif g' 
\end{align*}
by the $1$-Poincar\'e inequality \pik(1) again.
These calculations and \eqref{posi_p} yield 
\begin{equation*}
	J_1 \lsim	\int_{B(o,\A)}|\nabla f(g)| \dif g \lsim \int_{B(o,\A)}|\nabla f(g)| \,p(g) \dif g  \le \int_{\G}|\nabla f(g)| \,p(g) \dif g
\end{equation*}
as desired.

\subsection{Estimation of $J_2$} \label{sj2}
Let $\ss_\A:=\left\{g=(x,t)\in\ss; d(g) > \A \right\}$. Since the cut locus  $\mathrm{Cut}_o$, as well as the Carnot--Carath\'eodory sphere, has measure zero, we obtain
\begin{equation*}
	J_2= \int_{\ss_\A} \left| f(g)-f_{B(o,1)} \right| (1+d(g)) \, p(g) \dif g.
\end{equation*}
Next we put
\begin{equation} \label{est_N}
	N(g):= \lfloor 2\A^{-1}d(g)^3 \rfloor,\quad g\in\G.
\end{equation}
In the sequel we will always assume that $g\in\ss_{\A}$. Define 
\begin{gather*}
	s(i):= 1- i \,d(g)^{-3}, \quad i\in\nn\cap[0, 10 N(g)],\\
	g(i)=(x(i),t(i)) :=\gz(s(i)),
\end{gather*}
where $\gz: [0,\ 1] \to \G$ is the unique shortest geodesic joining $o$ to $g$.
By definition one has
\begin{equation} \label{d_si}
	d(g(i)) = s(i) d(g), \quad d(g(i),g(j)) = |i-j| d(g)^{-2}, \quad 0\le i,j\le 10 N(g).
\end{equation}

It can be easily derived that, as $\A$ is large enough, 
\begin{gather}
1 \ll \frac{d(g)}{2} \le  d(g(N(g))) + 4 d(g)^{-2}  = (1-N(g) d(g)^{-3}) d(g) +4 d(g)^{-2} \le (1-\A^{-1}) d(g). \nonumber
\end{gather}
Via the triangle inequality, we have 
\begin{align*}
\left|f(g)-f_{B(o,1)}\right| &\leq \left|f(g)-f_{B(g(N(g)), 2  d(g)^{-2})}\right| 
+ \, \left| f_{B(g(N(g)), 2  d(g)^{-2})} - f_{B(o,d(g(N(g))) + 4 d(g)^{-2})}\right|\\
&\qquad+ \left|  f_{B(o,d(g(N(g))) + 4 d(g)^{-2}) } - f_{B(o,1)}\right|.
\end{align*}
Thus we can bound the second term:
\begin{align*}
&\left| f_{B(g(N(g)), 2  d(g)^{-2})} - f_{B(o,d(g(N(g))) + 4 d(g)^{-2})}\right| \\
  &\qquad\le \,  \frac1{\vol(B(g(N(g)), 2  d(g)^{-2}))}\int_{B(g(N(g)), 2  d(g)^{-2})} |f(g')  - f_{B(o,d(g(N(g))) + 4 d(g)^{-2})}| \dif g'\\
  &\qquad\le \, \frac1{\vol(B(o, 2  d(g)^{-2}))} \int_{B(o,d(g(N(g))) + 4 d(g)^{-2})} |f(g')  - f_{B(o,d(g(N(g))) + 4 d(g)^{-2})}| \dif g' \\
  &\qquad\lesssim d(g)^{2Q+1}  \int_{B(o,d(g(N(g))) + 4 d(g)^{-2})} |\nabla f(g')| \dif g' \\
  &\qquad\le \,d(g)^{2Q+1}  \int_{B(o,(1-\A^{-1}) d(g))} |\nabla f(g')| \dif g',
\end{align*}
where we have used \pik(1) in the ``$\lesssim$''. Similarly, it holds that the third term
\begin{align*}
    \left|  f_{B(o,d(g(N(g))) + 4 d(g)^{-2}) } - f_{B(o,1)}\right|
   \lsim d(g)^{2Q+1}  \int_{B(o,(1-\A^{-1}) d(g))} |\nabla f(g')| \dif g'.
\end{align*}
Hence
\begin{align*}
	J_2 &\lsim \int_{\ss_\A} \left[ \int_{B(o,(1-\A^{-1}) d(g))} |\nabla f(g')| \dif g'\right]  d(g)^{2Q+ 2} p(g) \dif g \\
	&\qquad +  \int_{\ss_\A} \left|f(g)-f_{B(g(N(g)), 2  d(g)^{-2})}\right| d(g) p(g) \dif g =: J_{21} + J_{22}.
\end{align*}

\subsubsection{Estimate of $J_{21}$}
Noting that with  $M=M(g'):=(1-\A^{-1})^{-1} d(g')$,
\begin{align*}
&\left\{(g,g')\in\G\times\G; d(g) > \A, d(g') <  (1-\A^{-1}) d(g) \right\} \\
&\qquad\subset\left\{(g,g');  d(g') < \A  \right\} 
\cup \left\{(g,g'); d(g') \ge \A, d(g) \ge M \right\},
\end{align*}
then by \eqref{h_up} we have
\begin{align*}
	J_{21} &\lsim \int_{B(o, \A)}\left|\nabla f\left(g'\right)\right| \mathrm{d} g' \int_{\G} d(g)^{3Q+1} e^{-\frac{d(g)^2}{4}} \mathrm{~d} g \\
	& \qquad + \int_{B(o, \A)^c}\left|\nabla f\left(g'\right)\right|\left[\int_{B(o, M)^c} d(g)^{3Q+1 } e^{-\frac{d(g)^2}{4}} \mathrm{~d} g\right] \mathrm{d} g' =: J_{211} + J_{212}.
\end{align*}

The estimate of $J_{211}$ is easy since 
\begin{equation*}
	J_{211} \lsim \int_{B(o, \A)} |\nabla f(g')| \dif g' \sim  \int_{B(o, \A)} |\nabla f(g')| \,p(g') \dif g' \le  \int_\G |\nabla f(g')| \,p(g')\dif g',
\end{equation*}
where in the second step we have used \eqref{posi_p}.

To bound $J_{212}$, notice that $d(g')\ge\A\gg1$. We appeal to the polar coordinate formula (see, e.g., \cite[PROPOSITION 1.15]{FS82}). In fact, denoting by $\sigma(S_{\G})$ the surface area of unit sphere $S_{\G}:=\{g\in\G; \, d(g)=1\}$,  one sees that
\begin{align*}
	\int_{B(o, M)^c} d(g)^{3Q+1 } e^{-\frac{d(g)^2}{4}} \mathrm{~d} g &= \sigma(S_{\G}) \int_M^\infty  r^{4Q} e^{-\frac{r^2}4} \dif r
	=  2^{4Q }\sigma(S_{\G}) \int_{\frac{M^2}4}^\infty r^{\frac{4Q - 1}{2}} e^{-r} \dif r \\
 &=2^{4Q}\sigma(S_{\G}) \, \Gamma\left(\frac{4Q + 1}{2},\frac{M^2}{4}\right) \sim M^{\frac{4Q - 1}{2}} e^{-\frac{M^2}{4}},
\end{align*}
as long as we choose $\A$ to be large enough. Here $\Gamma(\cdot,\cdot)$ is the incomplete gamma function, whose asymptotic properties (see, e.g. \cite[\S~8.11.2]{P10}) have been used in the last step.  Hence the leftmost side of the above formula can be controlled by 
\begin{equation} \label{low_use}
C d(g')^{\frac{4Q - 1}2}\cdot e^{-\frac{d(g')^2}{4(1-\A^{-1})^2}} = C d(g')^{\frac{4Q - 1}2} e^{-\frac{\A^{-1}d(g')^2}{4(1-\A^{-1})^2}} \cdot e^{-\frac{d(g')^2}{4(1-\A^{-1})}} \lsim p(g')
\end{equation}
via the lower bound estimate \eqref{h_down} (with the choice $\dz=\frac1{\A}$), which implies that
\begin{equation*}
	J_{212} \lsim \int_{B(o,\A)^c} |\nabla f(g')| \,p(g') \dif g' \le  \int_\G |\nabla f(g')| \,p(g')\dif g'.
\end{equation*}
This, together with the estimate of $J_{211}$, yields the desired bound for $J_{21}$, i.e.,
\begin{equation*}
	J_{21} \lsim  \int_\G |\nabla f(g')| \,p(g')\dif g'.
\end{equation*}

\subsubsection{Estimate of $J_{22}$}
We are left to illustrate that $J_{22}$ also has such upper bound, that is,
\begin{equation*}
	J_{22} = \int_{\ss_\A} \left|f(g)-f_{B(g(N(g)), 2  d(g)^{-2})}\right| d(g) p(g) \dif g \lsim  \int_\G |\nabla f(g)| \,p(g)\dif g.
\end{equation*}

By the definition of $g(i)$ we have
\begin{equation*}
 g(i+1)\in B(g(i),2d(g)^{-2}),	\quad B(g(i),2d(g)^{-2}) \subset B(g(i+1),4d(g)^{-2}).
\end{equation*}
Then it follows from the triangle inequality and Lemma \ref{lem1} that
\begin{align*}
	|f(g) - f_{B(g(N(g)), 2  d(g)^{-2})}| &\le \sum_{i=0}^{N(g)-1} |f(g(i)) - f(g(i+1))| + |f(g(N(g))) - f_{B(g(N(g)), 2  d(g)^{-2})}| \\ 
	& \lsim \sum_{i=0}^{N(g)-1} \left[|f(g(i)) - f_{B(g(i),2d(g)^{-2})}| + |f(g(i+1)) -f_{B(g(i),2d(g)^{-2})} |  \right] \\
		&\qquad + \int_{B(g(N(g)),2d(g)^{-2})} d(g',g(N(g)))^{-Q+1} |\nabla f(g')| \dif g'\\
	&\lsim \sum_{i=0}^{N(g)-1} \left[ \int_{B(g(i),2d(g)^{-2})}  \frac{ |\nabla f(g')| \dif g'}{d(g',g(i))^{Q-1}} + \int_{B(g(i),2d(g)^{-2})}  \frac{ |\nabla f(g')| \dif g'}{d(g',g(i+1))^{Q-1}}	\right]\\
	 	&\qquad + \int_{B(g(N(g)),2d(g)^{-2})} 
	 	d(g',g(N(g)))^{-Q+1} |\nabla f(g')| \dif g'\\
	&  \lsim \sum_{i=0}^{N(g)} \int_{B(g(i),4d(g)^{-2})} d(g',g(i))^{-Q+1} |\nabla f(g')| \dif g'.
\end{align*}
 Consequently, 
\begin{equation*}
	J_{22} \lsim \int_{\ss_\A} \left[ \sum_{i=0}^{N(g)} \int_{B(g(i),4d(g)^{-2})} d(g',g(i))^{-Q+1} |\nabla f(g')| \dif g' \right]\, d(g)\, p(g) \dif g.
\end{equation*}

On the other hand, recalling \eqref{est_N}, we see for any $(g,g',i)\in\ss_\A\times\G\times\N$ satisfying $d(g',g(i))<4d(g)^{-2}$ and $i \le N(g)$
it holds 
\begin{align*}
	d(g') &\ge d(g(i)) - d(g',g(i)) \ge (1-i\,d(g)^{-3}) d(g) - 4d(g)^{-2} \\
	&\ge d(g) - (N(g)+4)d(g)^{-2} 
	> \frac12\A,
\end{align*}
and furthermore, 
$$d(g') \le d(g(i)) + d(g',g(i)) \le (1-i\,d(g)^{-3}) d(g) + 4d(g)^{-2} \le d(g) +4 d(g)^{-2},$$ 
provided that $\A$ is large enough. Then by \eqref{est_N} again we obtain
\begin{equation} \label{ddasy}
	\frac12\A < d(g')=d(g)(1+O(\A^{-1})), \quad \mbox{as}\,\,\A\to + \infty,
\end{equation}
which implies
\[
d(g',g(i))<10d(g')^{-2}, \qquad N(g) \le 2 N(g')
\]
for $\A$ large. As a result, choosing large $\A$ and an application of Fubini's theorem imply
\begin{equation} \label{mod1}
	J_{22} \lsim \int_{\overline{B(o,\frac12\A)}^{\,c}} 
 \left[ \sum_{i=0}^{2N(g')} \int_{\{g\in{\ss_\A}; \, d(g',g(i)) < 10d(g')^{-2}\}} d(g',g(i))^{-Q+1} p(g) \dif g \right]\,|\nabla f(g')| \,d(g')\dif g'.
\end{equation}
Thus to obtain the required upper bound of $J_{22}$ it suffices to prove that, for any $g'$ obeying $d(g') > \frac12\A$,
\begin{equation} \label{mod2}
	\sum_{i=0}^{2N(g')} \int_{\{g\in{\ss_\A}; \, d(g',g(i)) < 10d(g')^{-2}\}} d(g',g(i))^{-Q+1} p(g) \dif g =: \sum_{i=0}^{2N(g')} J_{22}(i)\lsim d(g')^{-1}p(g').
\end{equation}
For that, the following estimate will be enough:
\begin{equation}\label{j22i}
	 J_{22}(i) \lsim e^{-\frac{i}{8}d(g')^{-1}} d(g')^{-2} p(g'), \quad\mbox{when $0\le i\le 2N(g')$,  $d(g') > \frac12\A$,}
\end{equation}
since a summation on the geometric series leads to 
\begin{align*}
	\sum_{i=0}^{2N(g')} J_{22}(i) &\lsim \sum_{i=0}^{2N(g')} e^{-\frac{i}{8}d(g')^{-1}} d(g')^{-2} p(g') \le\sum_{i=0}^{+\infty} e^{-\frac{i}{8}d(g')^{-1}} d(g')^{-2} p(g') \\
	&\le \frac{d(g')^{-2} p(g')}{1-\exp\left\{-\frac{1}{8}d(g')^{-1}\right\}}\sim d(g')^{-1} p(g')
\end{align*}
by letting $\A\gg1$. 

To prove \eqref{j22i}, for any $i\in\nn\cap [0,2N(g')]$ we introduce the map $\Upsilon_i$ by setting	
\begin{equation*}
	\begin{aligned}
		\Upsilon_i: \ss_\A
		 &\rightarrow \Upsilon_i(\ss_\A)\subset\ss, \\
		 g=(x,t) &\mapsto \Upsilon_i(g) := g(i) = \gamma(s(i))=(x(i), t(i)),
	\end{aligned}
\end{equation*}
where we recall that $s(i)=s(i)(g)=1-i\, d(g)^{-3}$ and $\gz: [0,\ 1] \to \G$ denotes the unique shortest geodesic joining $o$ to $g$.

At this point it is quick to see that \eqref{j22i} can be indicated by the following lemma, whose proof will be postponed to Subsubsection \ref{coresec}.
\begin{lem} \label{core} 
	Let  
	\begin{gather} \label{L42A1}
	g\in\ss_\A, \quad	d(g') > \frac12\A,\quad d(g',g(i)) < 10 d(g')^{-2}, \quad i\in\nn\cap [0,2N(g')], 
		\end{gather}
and
	\begin{gather*}
		K_i(g):= \frac{p(g)\exp\left\{\frac{d(g)^2}{4}\right\}}{p(g(i))\exp\left\{\frac{d(g(i))^2}{4}\right\}}.
	\end{gather*}
Then by choosing suitable $\A$ we have
{\em\begin{compactenum}[(i)]
\item $\Upsilon_i$ is a $C^\infty$ diffeomorphism and $\jac(\Upsilon_i)>0$ on $\ss_{\A}$, for each $i$.		
\item There exist two universal constants $C_1,C_2>0$ such that
\begin{gather} 
	p(g)\le C_1 K_i(g) \, p(g') \, \exp\left\{-\frac{i}{8}d(g')^{-1} \right\},  \label{core1} \\
	K_i(g) \le C_2 \, \jac(\Upsilon_i)(g). \label{core2}
\end{gather}
\end{compactenum}}
\end{lem}
 Indeed, through \eqref{core1} and the change of variables $g = \Upsilon_i^{-1}(g(i)) $ we have
\begin{align*}
	J_{22}(i) &\lsim p(g') e^{-\frac{i}{8}d(g')^{-1} } \int_{\{g\in\ss_\A; \, d(g',g(i)) < 10d(g')^{-2}\}} d(g',g(i))^{-Q+1} K_i(g) \dif g \\
	&= p(g') e^{-\frac{i}{8}d(g')^{-1} } \int_{\{g(i); \, \Upsilon_i^{-1}(g(i))\in\ss_\A, d(g',g(i)) < 10d(g')^{-2}\}} d(g',g(i))^{-Q+1} \frac{K_i(\Upsilon_i^{-1}(g(i)))}{\jac(\Upsilon_i)(\Upsilon_i^{-1}(g(i)))} \dif g(i),
\end{align*}
which, together with \eqref{core2} and the change of variable $(g')^{-1} \cdot g(i)\mapsto \tilde{g}$  yields
\begin{align*}
	J_{22}(i) &\lsim p(g') e^{-\frac{i}{8}d(g')^{-1} }\int_{B(g',10d(g')^{-2})} d(g',g(i))^{-Q+1} \dif g(i) \\
	&= p(g') e^{-\frac{i}{8}d(g')^{-1} }\int_{B(o,10d(g')^{-2})} d(\tilde{g})^{-Q+1} \dif \tilde{g} \\
	&\lsim p(g') e^{-\frac{i}{8}d(g')^{-1} } \int_0^{10d(g')^{-2}} dr \sim p(g') e^{-\frac{i}{8}d(g')^{-1} } d(g')^{-2},
\end{align*}
where in the third step we have used the polar coordinate. This concludes the proof of \eqref{j22i}. 

Finally, it remains to provide the
\subsubsection{Proof of Lemma \ref{core}} \label{coresec}
Supposing condition \eqref{L42A1} holds, we first prove item (i). From Proposition \ref{propgeo}, it turns out that
\begin{equation*}
	\Upsilon_i  = \exp \circ\, \Gz_i \circ \exp^{-1}
\end{equation*}
with the map $\Gz_i: \mathcal{D} \to \mathcal{D}$  defined by
\begin{equation*}
	\begin{aligned}
		\Gz_i: \mathcal{D} 
		&\rightarrow \mathcal{D}, \\
		(\zeta,\tau) &\mapsto \Gz_i(\zeta,\tau) := (1-i\, |\zeta|^{-3})(\zeta,\tau).
	\end{aligned}
\end{equation*}

According to the chain rule we have
\begin{equation} \label{chain1}
	\jac(\Upsilon_i)(x,t) = \jac(\exp)(\Gz_i(\zeta,\tau)) \jac(\Gz_i)(\zeta,\tau) (\jac(\exp)(\zeta,\tau))^{-1}
\end{equation}
with $(\zeta,\tau) = \exp^{-1}(x,t)$. On the other hand, by \eqref{est_N} and \eqref{d_si} it is easy to see that if $\A$ is large enough, then
\begin{equation} \label{e101}
	\frac12 \le s(i) \le1 \quad\mbox{and}  \quad d(g(i)) =d(g)(1+O(\A^{-1})) =d(g') \gsim \A. 
\end{equation}
A direct calculation shows that the Jacobian matrix of $\Gz_i(\zeta,\tau)$ equals
\begin{align*}
3i|\zeta|^{-5}\begin{pmatrix} \zeta \\ \tau \end{pmatrix} (\zeta^\transpose,0) +  (1-i\, |\zeta|^{-3}) \,\I_{q+m}  = 3 i\,|\zeta|^{-5} \begin{pmatrix}
      \zeta \zeta^\transpose & 0 \\ \tau\zeta^\transpose &   0
    \end{pmatrix} + (1-i\, |\zeta|^{-3})\,\I_{q+m} .
\end{align*}
Hence by letting $\A\gg1$ it yields that
\begin{equation}\label{chain2}
	\jac(\Gz_i)(\zeta,\tau) = (1-i\, |\zeta|^{-3})^{q + m - 1}(1  + 2 i\,|\zeta|^{-3}) \sim 1
\end{equation}
for any $(\zeta,\tau) = \exp^{-1}(x,t)$ with $(x,t) \in \ss_{\A}$. 

Recall that from \cite[Corollary 6.8]{GZ24} we have
\begin{equation} \label{jac0}
	\jac(\exp)(\zeta,\tau) >0, \quad \forall\, (\zeta,\tau)\in\mathcal{D}.
\end{equation}
These then imply that $\jac(\Upsilon_i) > 0$ on $\ss_\A$ and $\Upsilon_i$ is a $C^\infty$ diffeomorphism.

We next show that there is a universal constant $C>0$ such that
	\begin{equation} \label{L42eq1}
	    p(g(i)) \le C p(g').
	\end{equation}
In fact, by \cite[Theorem 20.3.1 with the footnote on p. 747]{BLU07}  we have 
\begin{equation} \label{mean1}
    |\ln{p(g(i))} - \ln{p(g')}| \le  d(g(i),g') \sup\limits_{g_*\in B\left(g',d(g(i),g')\right)} |\nabla \ln{p(g_*)}|.
\end{equation}
Since as $g_*\in B\left(g',d(g(i),g') \right)\subset B\left(g',10 d(g')^{-2}\right)$ it holds $d(g_*) \sim d(g')$, then by the gradient estimate \eqref{assud}
\[
|\ln{p(g(i))} - \ln{p(g')}| \lsim d(g(i),g') \sup\limits_{g_*\in B\left(g',d(g(i),g')\right)} d(g_*) \lesssim d(g')^{-1} \lesssim 1,
\]
which gives \eqref{L42eq1}. Notice that by \eqref{d_si} and \eqref{e101} 
\begin{equation*}
    d(g(i))^2-d(g)^2 = (s(i)-1)(s(i)+1)d(g)^2\le -i\,d(g)^{-1}\le -\frac{i}2d(g')^{-1}.
\end{equation*}
Combining this with \eqref{L42eq1} we obtain \eqref{core1}.

Now with \eqref{chain1}, \eqref{chain2},  Assumption \ref{ass} and the aid of Proposition \ref{HCP} (ii) below,  one can deduce \eqref{core2} immediately by plugging $\eta = (\zeta,\tau) = \exp^{-1}(g)$ and $s = 1-i\, |\zeta|^{-3} \in[\frac12,1]$ into \eqref{mcpcheck}.

\begin{prop} \label{HCP}
Assume that $N>0$. We have:
{\em\begin{compactenum}[(i)]
\item The statement \eqref{defmcp}
is equivalent to
\begin{equation} \label{eq_mcp}
	\jac(\exp)(s \eta) \ge  s^{N - q - m} \jac(\exp)(\eta), \qquad \forall\, s\in(0,1], \, \eta \in \mathcal{D}.
\end{equation}
\item Write $h:=p\,e^{\frac{d^2}4}$ and $H:= h\circ \exp$. Then \eqref{assmcp} is equivalent to the following statement  with the same constant $C$:		
\begin{align}\label{mcpcheck}
H(s \eta) \, \jac(\exp)(s \eta) \ge C^{-1} s^{N - q - m} H(\eta)\, \jac(\exp)(\eta), \qquad \forall\, s\in(0,1], \, \eta \in \mathcal{D}.
\end{align}
\end{compactenum}}		
\end{prop}

\begin{proof}
The proof of Proposition \ref{HCP} is easy and we only prove item (ii) since item (i) follows in a similar manner (see also \cite[Proposition 5.1 and Corollary 6.8]{GZ24} for item (i)). Suppose \eqref{assmcp} holds. Put $\mathscr{B}_\ss:=\{E\in\mathscr{B}; E\subset\ss\}$. Then for every $E \in \mathscr{B}_\ss$ we have
\[
\int_{Z_s(o,E)} h(g) \dif g \ge C^{-1} s^N \int_{E} h(g) \dif g.
\]
Notice that for any $g'' \in E \subset \mathcal{S}$, the only point satisfying 
$$d(o,\cdot) = s\, d(o,g''),\quad d(g'',\cdot) = (1-s) \,d(o,g'')$$
 is exactly $\gamma(s)$, where $\gz: [0,\ 1] \to \G$ is the unique shortest geodesic joining $o$ to $g''$. Then it follows from  a change of variable $g = \exp(\eta)$, \eqref{jac0} and Proposition \ref{propgeo}  that
\[
\int_{s \exp^{-1}(E)} H(\eta) \,\jac(\exp)(\eta) \dif \eta \ge C^{-1} s^N \int_{\exp^{-1}(E)} H(\eta)\,  \jac(\exp)(\eta) \dif \eta.
\]
Another change of variable $\eta \mapsto s \eta$ on the left side indicates
\[
s^{q + m} \int_{\exp^{-1}(E)} H(s \eta)\, \jac(\exp)(s \eta) \dif \eta \ge C^{-1} s^N \int_{\exp^{-1}(E)} H(\eta)\, \jac(\exp)(\eta) \dif \eta.
\]
Since $E$ is arbitrary, we conclude that \eqref{mcpcheck} holds. The opposite implication is now obvious for any $E \in \mathscr{B}_\ss$. This yields \eqref{assmcp} at once by the fact that $\ss^c$ is of measure zero. Indeed, for $\forall\, s\in(0,1], \, E\in\mathscr{B}$, we have $E\cap\ss\in\mathscr{B}_\ss$ and hence
\[
C^{-1} s^{N} \mu(E) = C^{-1} s^{N} \mu(E\cap\ss)\le \mu(Z_s(o,E\cap\ss)) \le \mu(Z_s(o,E)).
\]

\end{proof}

\section{Proofs of Theorems \ref{thm2} and \ref{thm3}}\label{s32}
\setcounter{equation}{0}

Suppose Assumptions \ref{ass} and \ref{ass2} hold. We first prove the following lemma, which is useful in the proof of Theorem \ref{thm2}.

\begin{lem}\label{lemr}
Let $k\in\nn^*$. Then there is a positive constant $C_k$ such that, for any $f\in C^\infty_\mathrm{pol} (\G)$,
\begin{equation} \label{eq1}
     \int_\G|f(g)| d(g)^k\,p(g) \dif g \le C_k \int_\G|\nabla^k f(g)| \, p(g) \dif g +C_k  \int_\G|f(g)|\, (1+d(g))^{k-1} \,p(g) \dif g.
\end{equation}
Here $C^\infty_\mathrm{pol}(\G)$ denotes the set of smooth functions defined on $\G$ with polynomial  growth at infinity:
\begin{equation*}
    C^\infty_\mathrm{pol} (\G):=\{f\in C^\infty; \, \forall\, k\in\nn, \exists\, D_k, m_k>0, \mbox{s.t.}\, |\nabla^kf(g)|\le D_k (1+d(g))^{m_k}, \forall\, g\in\G\}.
\end{equation*}
\end{lem} 

\begin{proof}
We only consider that $f\in C^\infty_c(\G)$, since general case can be reached via a simple limiting argument. We first prove the case $k = 1$. In fact this is a consequence of \eqref{hsg1}. More precisely, it follows from \eqref{hsg1} that 
\begin{equation*}
	\int_{\G} \left| f(g)-f_{B(o,1)} \right| (1+d(g)) \, p(g) \dif g \lsim \int_{\G}|\nabla f(g)| \,p(g) \dif g
\end{equation*}
for $f\in C^\infty_c(\G)$. Moreover, by \eqref{posi_p},
\begin{equation*}
	|f_{B(o,1)}| = \frac{1}{\vol(B(o,1))}\left|\int_{B(o,1)} f(g)\dif g \right| \lsim \int_{B(o,1)} |f(g)|\,p(g) \dif g \le \int_\G | f(g)| \, p(g)\dif g.
\end{equation*}
A combination of these  two estimates  above with the triangle inequality leads to the desired estimate for $f\in C^\infty_c(\G)$ (noticing that $\int (1+d) \, p < \infty$, since $p$ is a Schwartz function and $d$ is of polynomial growth (see \eqref{ehd})). 

Also notice that in the proof in Subsection \ref{sJ1}, fixing an $\A$ in the proof there, we have actually proved that
\begin{align*}\label{eq77}
\int_{B(o,\A)} \left| f(g)-f_{B(o,1)} \right| \,p(g) \dif g \lsim \int_{B(o,\A)} \left| f(g)-f_{B(o,1)} \right| \dif g \lsim \int_{\G}|\nabla f(g)| \, (1 + d(g))^{-1} \,p(g) \dif g.
\end{align*} 

On the other hand, using the argument for the estimate of $J_2$ in Subsection \ref{sj2}, we obtain
\begin{equation*}
	\int_{B(o,\A)^c} \left| f(g)-f_{B(o,1)} \right| \,p(g) \dif g  \lsim  \int_{\G}|\nabla f(g)|\, (1 + d(g))^{-1} \,p(g)  \dif g.
\end{equation*}
In fact, the main key points are as follows: the estimate for ``$J_{211}$'' is similar since on $B(o,\A)$ we have $1 \sim (1 + d(g'))^{-1}$; for ``$J_{212}$" part one can use $d(g')^{-1} p(g')$  in  the last bound of \eqref{low_use} instead of $p(g')$; while for ``$J_{22}$" part, the measure  $d(g')\dif g'$ in the outer integration of \eqref{mod1} now should be replaced by $\dif g'$ (notice that \eqref{mod2} is still valid). The rest of the proof can also be checked with minor modifications.

As a result we have
\begin{equation}\label{eq76}
	\int_{\G} \left| f(g)-f_{B(o,1)} \right| \,p(g) \dif g  \lsim  \int_{\G}|\nabla f(g)| \, (1 + d(g))^{-1}\,p(g) \dif g.
\end{equation}

As before, by \eqref{posi_p},
\begin{equation*}
	|f_{B(o,1)}| = \frac{1}{\vol(B(o,1))}\left|\int_{B(o,1)} f(g)\dif g \right| \lsim \int_{B(o,1)} |f(g)|\,p(g) \dif g \le 2\int_\G | f(g) | (1 + d(g))^{-1} \, p(g)\dif g.
\end{equation*}
We thus get a similar inequality:
\begin{align}\label{eq77}
\int_\G|f(g)| \,p(g) \dif g \lesssim  \int_\G|\nabla f(g)| \, (1 + d(g))^{-1} \, p(g) \dif g +   \int_\G|f(g)|\, (1+d(g))^{-1} \,p(g) \dif g.
\end{align}

Now we can use an induction argument to prove the case $k\ge2$. Suppose that \eqref{eq1} is correct for any  $m\in \{1, \ldots, k-1\}$. That is, for any  $f\in C^\infty_c(\G)$, it holds that
\begin{equation} \label{eq4}
	\int_\G |f(g)| d(g)^m \, p(g) \dif g \lsim_m \int_\G |\nabla^m f(g)|\, p(g) \dif g +\int_\G |f(g)| (1+d(g))^{m-1} p(g) \dif g.
\end{equation}
Let $P$ be any homogeneous polynomial of order $2$. Then if $k > 2$ by induction we have
\begin{align*}
\int_\G |f(g) P(g)| d(g)^{k - 2} \, p(g) \dif g &\lsim_k \int_\G  |\nabla^{k - 2} f(g)| d(g)^2 \, p(g) \dif g + \int_\G  |\nabla^{k - 3} f(g)| d(g) \, p(g) \dif g\\
&\qquad + \int_\G  |\nabla^{k - 4} f(g)|  \, p(g) \dif g + \int_\G f(g) (1 + d(g))^{k - 1} \, p(g) \dif g,
\end{align*}
where  $|\nabla^{k - 4} f(g)|$ is interpreted as $0$ when $k = 3$.  For $k = 2$, using \eqref{eq77} instead we obtain
\begin{align*}
\int_\G |f(g) P(g)|  \, p(g) \dif g \lsim \int_\G  |\nabla f(g)| d(g) \, p(g) \dif g 
+ \int_\G f(g) (1 + d(g)) \, p(g) \dif g.
\end{align*}

By choosing $P$ to be $x_j^2, 1 \le j \le q$ and $t_\ell, 1 \le \ell \le m$ in turn, adding these estimates together and applying \eqref{ehd}, we can replace $P(g)$ with $d(g)^2$ in the estimates above (up to a uniform constant). For the estimates on the right side, we iterate the argument above and conclude that:
\begin{equation}\label{eq7}
 \begin{aligned} 
	\int_\G |f(g)| d(g)^k p(g) \dif g &\lsim_k \int_\G |\nabla^k f(g)|\,p(g)\dif g + \sum_{i=1}^{k-1} \int_\G |\nabla^i f(g)|\, p(g) \dif g  \\
 &\qquad + \int_\G | f(g)| (1+d(g))^{k-1} p(g) \dif g.
\end{aligned}
\end{equation}
Regarding this iteration argument, for instance, if $k = 2$, from the argument above we have 
\begin{align}\label{ke2}
\int_\G |f(g) | d(g)^2  \, p(g) \dif g \lsim \int_\G  |\nabla f(g)| d(g) \, p(g) \dif g 
+ \int_\G f(g) (1 + d(g)) \, p(g) \dif g.
\end{align}
Next we bound $|\nabla f|$ in the first integral by $\sum_{\ell = 1}^m |\X_\ell f|$. For each $\ell$, from induction assumption we can apply \eqref{eq4} with $m = 1$ to $\X_\ell f$ and  it turns out that 
\[
\int_\G  |\X_\ell f(g)| d(g) \, p(g) \dif g \lesssim \int_\G |\nabla^2 f(g)|\, p(g) \dif g +\int_\G |\nabla f(g)|  p(g) \dif g.
\]
Inserting this into \eqref{ke2} we obtain \eqref{eq7} for $k = 2$. General cases follow in a similar way.

Compared with our target \eqref{eq1}, one sees that only the second term in the right side of \eqref{eq7} needs handling further.

To deal with these integrals, given $f\in C^\infty_c (\G)$, we put
	\begin{equation*}
		\mathbb{P}(f):= \int_\G f(g) p(g) \dif g.
	\end{equation*} 
	Then for any complex number $a$ it holds that $\mathbb{P}(|f-\PP(f)|)\le 2 \, \PP(|f-a|)$ owing to the stochastic completeness. From this and \eqref{eq76}, there is constant $C>0$ such that,
	\begin{equation} \label{eq101}
		\int_\G |f(g)- \PP(f)|\, p(g) \dif g \le C \int_\G |\nabla f(g)| \, p(g) \dif g, \quad \forall\, f \in C^\infty_c (\G).
	\end{equation}

Using \eqref{eq101} we have that, for any  $f\in C^\infty_c(\G)$ and $I$ satisfying $1\le |I|' \le k-1$,
\begin{equation*} 
	\int_\G |\X^I f(g)- \PP(\X^I f)|\, p(g) \dif g \lsim \int_\G |\nabla^{|I|'+1}f(g)|\, p(g) \dif g.
\end{equation*} 
Moreover, denoting by $\Y^{I} := (-1)^{|I|'} \, \X_{i_{|I|'}}\cdots \X_{i_1}$ the adjoint operator of $\X^I$, an integration by parts together with \eqref{assud} and Assumption \ref{ass2} indicates  that
\begin{equation*}
	\left| \PP(\X^I f)\right| = \left| \int_\G \X^I f(g) \, p(g) \dif g \right| =  \left| \int_\G f(g) \,\Y^I p(g) \dif g \right| \lsim_k \int_\G |f(g)| (1+d(g))^{|I|'} \, p(g) \dif g.
\end{equation*}
These two estimates and the triangle inequality then give 
\begin{align*}
	 \int_\G  \left|\X^I f(g) \right| \, p(g) \dif g 
	 &\le \int_\G |\X^I f (g)- \PP(\X^I f)|\, p(g) \dif g + \left| \PP(\X^I f)\right| \\
	 &\lsim_k \int_\G |\nabla^{|I|'+1}f (g)|\, p(g) \dif g + \int_\G |f(g)| (1+d(g))^{|I|'} \, p(g) \dif g.
\end{align*}
As a consequence,  for $1\le i \le k-1$, summing over $I$ where $|I|'=i$ implies that,
\begin{align*}
	\int_\G  \left|\nabla^i f(g) \right| \, p(g) \dif g 
	&\lsim_i \int_\G |\nabla^{i+1}f(g)|\, p(g) \dif g + \int_\G |f(g)| (1+d(g))^i \, p(g) \dif g.
\end{align*}
Iterating this estimate produces that,  for $1\le i \le k-1$,
\begin{align*}
	\int_\G  \left|\nabla^i f(g) \right| \, p(g) \dif g 
	&\lsim_k \int_\G |\nabla^{k}f(g)|\, p(g) \dif g + \sum_{j=i}^{k-1}\int_\G |f(g)| (1+d(g))^i \, p(g) \dif g \\
	&\lsim_k \int_\G |\nabla^{k}f(g)|\, p(g) \dif g + \int_\G |f(g)| (1+d(g))^{k - 1} \, p(g) \dif g.
\end{align*}
From this and \eqref{eq7}, it follows that \eqref{eq1} holds. Then the proof of Lemma \ref{lemr} is completed.

\end{proof}

\subsection{Proof of Theorem \ref{thm2}}\label{ss321}

As before, it suffices to prove \eqref{hsg2} for $h=1$ and $g = o$. In other words, we only have to show that, for any $I$ of length $k$ (that is, $|I|' = k$),
\begin{equation*}
    \left| \int_\G [f(g)- P_k(B(o,\B),f)(g)]\, \X^I p(g^{-1}) \dif g\right| \lsim_k \int_\G |\nabla^k f(g)| \, p(g) \dif g, \quad \forall\, f\in C_c^\infty(\G),
\end{equation*}
where $\B=\B(k)>1$ is a large constant to be chosen later. According to this, \eqref{assud} and Assumption \ref{ass2}, we are left to show that 
\begin{equation} \label{red1}
    \int_\G  \left| f(g)- P_k(B(o,\B),f)(g) \right| \,(1+d(g))^k \dif g  \lsim_k \int_\G |\nabla^k f(g)| \, p(g) \dif g, \quad \forall\, f\in C^\infty_c(\G).
\end{equation}

To see this, given $f\in C^\infty_c (\G)$, we recall that $P_k(B(o,\B),f)$ is a polynomial with homogeneous order less than $k$. Then by \eqref{eq1},
\begin{equation} \label{eq2}
\begin{aligned} 
	& \int_\G |f(g)- P_k(B(o,\B),f)(g)|\, d(g)^k p(g) \dif g 
	 \\
	 &\qquad \le C_k \int_\G|\nabla^k f(g)| \, p(g) \dif g +C_k  \int_\G|f(g)- P_k(B(o,\B),f)(g)|\, (1+d(g))^{k-1}\,p(g) \dif g.
\end{aligned}
\end{equation}
Here we note that the constant $C_k$ is independent of $\B$.
Hence \eqref{red1} can be proved if
\begin{equation} \label{eq3}
	 \int_\G|f(g)- P_k(B(o,\B),f)(g)|\, (1+d(g))^{k-1} p(g) \dif g \lsim_k \int_\G |\nabla^k f(g)| \, p(g) \dif g.
\end{equation} 

We split the last integral into the one on $B(o,\B)$ and the one on $B(o,\B)^c$. Then from \eqref{posi_p} and Lemma \ref{lem2} (\pik($k$)) it implies that 
\begin{align*}
	 &\int_{B(o,\B)} |f(g)- P_k(B(o,\B),f)(g)|\, (1+d(g))^{k-1} p(g) \dif g \\
	 &\qquad\lsim_{k,\B}  \int_{B(o,\B)}|f(g)- P_k(B(o,\B),f)(g)| \dif g \lsim_{k,\B}   \int_{B(o,\B)} |\nabla^k f(g)| \dif g \\
	&\qquad \lsim_{\B}  \int_{B(o,\B)} |\nabla^k f(g)| \, p(g) \dif g \le   \int_{\G} |\nabla^k f(g)| \, p(g) \dif g.
\end{align*}
On the other hand, we have
\begin{align*}
	&\int_{B(o,\B)^c} |f(g)- P_k(B(o,\B),f)(g)|\, (1+d(g))^{k-1} p(g) \dif g  \\
\le \, & 2^{k - 1}\int_{B(o,\B)^c} |f(g)- P_k(B(o,\B),f)(g)|\, d(g)^{k-1} p(g) \dif g  \\
\le \, & \frac{2^{k - 1}}{\B} \int_\G |f(g)- P_k(B(o,\B),f)(g)| \, d(g)^k p(g) \dif g \\
\le \, & \frac{2^{k - 1}C_k}{\B} \int_\G|\nabla^k f(g)| \, p(g) \dif g + \frac{2^{k - 1}C_k}{\B}  \int_\G|f(g)- P_k(B(o,\B),f)(g)|\, (1+d(g))^{k-1} \,p(g) \dif g,
 \end{align*}
 where in the last step we have used \eqref{eq2}. Combining the above two estimates we see that there is a constant $C'(k,\B)$ such that,
\begin{align*}
	&\int_{\G} |f(g)- P_k(B(o,\B),f)(g)|\, (1+d(g))^{k-1} p(g) \dif g  \\
\le  \, & C'(k,\B) \int_\G|\nabla^k f(g)| \, p(g) \dif g + \frac{2^{k - 1}C_k}{\B}  \int_\G|f(g)- P_k(B(o,\B),f)(g)|\, (1+d(g))^{k-1} \,p(g) \dif g.
\end{align*}
Consequently, by choosing $\B = 2^{k - 1}C_k+1$, 
\begin{equation*}
	\int_{\G} |f(g)- P_k(B(o,\B),f)(g)|\, (1+d(g))^{k-1} p(g) \dif g  \le ( 2^{k - 1}C_{k}+1) C'(k, 2^{k - 1}C_{k}+1) \int_\G|\nabla^k f(g)| \, p(g) \dif g,
\end{equation*}
which yields \eqref{eq3} and hence Theorem \ref{thm2}.

\subsection{Proof of Theorem \ref{thm3}}\label{ss322}
The proof of  Theorem \ref{thm3} is in the same spirit of \cite[Section 6]{Z23}. In fact, noting that $\T_i$ commutes with $\T_j$ and $\X_k$ for arbitrary $i,j,k$, we are required to prove that
\begin{equation} \label{eq21}
	\left|\X^{I_1}\T^{I_2} e^{h \Delta_\RR} f\right|(g) \lsim_k e^{h \Delta_\RR}(|\nabla_\RR^k f|)(g),
\end{equation}
for any $I_1,I_2$ with $|I_1|' + |I_2|' = k$.
Moreover, we can check $e^{h\Dz_{\RR}} = e^{h\Dz}e^{h\Dz_{\T}}$, $\X_i$ commutes with $e^{h\Dz_{\T}}$, and $\T_\ell$ commutes with $e^{h\Dz_{\RR}}$. Hence from Theorem \ref{thm2} we obtain 
\begin{align*}
	\left|\X^{I_1}\T^{I_2} e^{h \Delta_\RR} f\right|(g) &=  |\X^{I_1}e^{h\Dz}\left(e^{h\Dz_{\T}}\T^{I_2}f\right)|(g) \\
	&\lsim_k e^{h\Dz} \left( |\nabla^{k_1} e^{h\Dz_{\T}}\T^{I_2}f | \right) (g) \lsim_k e^{h\Dz}\left( e^{h\Dz_{\T}} (| \nabla^{k_1} \T^{I_2}f |) \right) (g) \\
	 &\lsim_k  e^{h\Dz}\left( e^{h\Dz_{\T}} (| \nabla^{k}_\RR f |) \right) (g) = e^{h\Dz_{\RR}}\left( | \nabla^{k}_\RR f | \right) (g).
\end{align*}	
This completes the proof of \eqref{eq21}.

\section{Examples and applications of the main results} \label{s4}
In this Section we furnish some examples which satisfy Assumptions \ref{ass} and \ref{ass2}, including Heisenberg groups, H-type groups and the free step-two Carnot group with three generators. As a result, we obtain Theorems \ref{thm1}-\ref{thm3} on those groups.

\subsection{The Heisenberg group $\H^n$} \label{s41}

Recall that the Heisenberg group $\H^n \cong \R^{2n} \times \R$ is defined by the following group multiplication:
\[
(x,t) \cdot (x',t') = \left(x + x', t + t' + \frac{1}{2} \sum_{j = 1}^n  (x_{2j - 1}x_{2j}' - x_{2j} x_{2j - 1}' ) \right).
\]

Now we fix an $n \in \N^*$. Since Assumption \ref{ass2} follows from \cite[Theorem 2]{Li07}, we are left with Assumption \ref{ass}, and by Proposition \ref{HCP} (ii) we only need to prove that \eqref{mcpcheck} holds for some $N$ with a universal constant $C$. In fact, it yields from the computation in \cite{J09} that $\mathcal{D} = (\R^{2n} \setminus \{0\}) \times (-2\pi, 2\pi)$ and if we write $\eta = (\zeta,\tau) = (\zeta, 2 \theta) \in \mathcal{D}$, then
\[
\jac(\exp)(\zeta, 2 \theta) \sim |\zeta|^2 \left( \frac{\sin{\theta}}{\theta} \right)^{2n - 1}, 
 \qquad |x| = \frac{\sin{\theta}}{\theta} |\zeta|.
\]
Combining this with the precise bounds of the heat kernel (see \cite[Theorem 1]{Li07})  
\begin{align}\label{pulH}
p(x,t) \sim \frac{(1 + d(g))^{2n - 2}}{(1 + |x|d(g))^{n - \frac{1}{2}}} \, e^{-\frac{d(g)^2}{4}},
\end{align}
we obtain 
\[
H(\zeta, 2 \theta) \, \jac(\exp)(\zeta, 2 \theta)\sim \begin{cases} \left( \frac{\sin{\theta}}{\theta} \right)^{2n - 1}  |\zeta|^2 , & \text { if } |\zeta| \lesssim 1,
\\ 
\left( \frac{\sin{\theta}}{\theta} \right)^{n - \frac{1}{2}} |\zeta|, & \text { if } |\zeta| \gg 1, \frac{\sin{\theta}}{\theta} |\zeta|^2 \gg 1, \\
\left( \frac{\sin{\theta}}{\theta} \right)^{2n - 1} |\zeta|^{2n}, & \text { if } |\zeta| \gg 1, \frac{\sin{\theta}}{\theta} |\zeta|^2 \lesssim 1.
\end{cases}
\]
Noticing that the function $s \mapsto \frac{\sin{s}}{s}$ is an even function on $(-\pi,\pi)$ and strictly decreasing on $[0,\pi)$, it is not hard to check \eqref{mcpcheck} holds uniformly on $(0,1]$, with $N=4n+1$ by the formulae above.

\subsection{The H-type group $\H(2n,m)$} \label{s43}

If the matrices $\U = \{U^{(1)},\ldots,U^{(m)}\}$ appearing in \eqref{gst} are skew-symmetric and orthogonal matrices satisfying $ U^{ (i) } U^{ (j) } = -U^{ (j) } U^{ (i) }$ for all $1 \le  i \neq j \le m$, then we call the group    H-type group. It turns out that $q = 2n$ and $m < \rho(2n)$, where the Hurwitz--Radon function $\rho$ is defined by
	\begin{align*}
		\rho(2n) := 8 k + 2^l, \quad \mbox{with}
		\quad 2n = \mbox{(odd)} \cdot 2^{4k + l}, \ 0 \leq l \leq 3, \ k = 0, 1, \ldots,
	\end{align*}
	which implies $2n \ge m + 1$. See for example \cite{K80} for more details. In this work we use $\mathbb{H}(2n,m)$ to denote such H-type group.

Fix $n,m \in \N^*$. Since Assumption \ref{ass2} follows from \cite[Theorem 1.2]{Li10}, we are left with Assumption \ref{ass}, and by Proposition \ref{HCP} (ii) we only need to prove that \eqref{mcpcheck} holds for some $N$ with a universal constant $C$. 
In fact, it yields from the computation in \cite{BR18} that $\mathcal{D} = (\R^{2n} \setminus \{0\}) \times B_{\R^m}(0,2\pi)$ and if we write $\eta = (\zeta,\tau) = (\zeta, 2 \theta) \in \mathcal{D}$, then
\[
\jac(\exp)(\zeta, 2 \theta) \sim |\zeta|^{2m} \left( \frac{\sin{|\theta|}}{|\theta|} \right)^{2n - 1},
 \qquad |x| = \frac{\sin{|\theta|}}{|\theta|} |\zeta|.
\]
Combining this with the precise bounds of the heat kernel (see \cite[Theorem 1.1]{Li10} or \cite[Theorem 4.2]{E09})  
\begin{align}\label{pulHt}
p(x,t) \sim \frac{(1 + d(g))^{2n -m - 1}}{(1 + |x|d(g))^{n - \frac{1}{2}}} \, e^{-\frac{d(g)^2}{4}},
\end{align}
we have
\[
H(\zeta, 2 \theta) \, \jac(\exp)(\zeta, 2 \theta)\sim \begin{cases} \left( \frac{\sin{|\theta|}}{|\theta|} \right)^{2n - 1}  |\zeta|^{2m} , & \text { if } |\zeta| \lesssim 1,
\\ 
\left( \frac{\sin{|\theta|}}{|\theta|} \right)^{n - \frac{1}{2}} |\zeta|^m, & \text { if } |\zeta| \gg 1, \frac{\sin{|\theta|}}{|\theta|} |\zeta|^2 \gg 1, \\
\left( \frac{\sin{|\theta|}}{|\theta|} \right)^{2n - 1} |\zeta|^{2n + m - 1}, & \text { if } |\zeta| \gg 1, \frac{\sin{|\theta|}}{|\theta|} |\zeta|^2 \lesssim 1.
\end{cases}
\]
Similar to the Heisenberg group case, we can check \eqref{mcpcheck} holds uniformly on $(0,1]$, with $N=4n + 2m - 1$ by the formulae above.

\subsection{The non-isotropic Heisenberg group and generalized H-type group} \label{s44}

Using the arguments in \cite[\S~7.2, 7.4]{Z23} we know that \eqref{hsge} holds on non-isotropic Heisenberg groups and generalized H-type groups. 

In fact, one can check that Assumptions \ref{ass} and \ref{ass2} are both valid on those groups. We refer the interested reader to the forthcoming paper \cite{LZ22} for more details.

\subsection{The free step-two Carnot group with three generators $N_{3,2}$ } \label{s42}

The free step-two Carnot group with three generators $N_{3,2}$ is $\rr^3 \times \rr^3$ with the following group structure

\[
(x,t) \cdot (x',t') = \left(x + x' , t + t' - \frac{1}{2} \, x \times x' \right),
\]
where ``$\times$'' represents the cross product on $\rr^3$, i.e.,
\[
x \times x' = (x_2 x_3' - x_3 x_2', x_3 x_1' - x_1 x_3', x_1 x_2' - x_2 x_1').
\]

Assumption \ref{ass2} follows from \cite[Theorem 12.1]{LMZ23}. To show that Assumption \ref{ass} holds on $N_{3,2}$, we need to first recall some known results.

From \cite[Corollary 16]{LZ21} it follows that
$$\ss=\{(x,t); \, x \,\, \mbox{and}\,\, t \,\,\mbox{are linearly independent}\}=\left\{(x,t); |x|^2|t|^2\neq (x\cdot t)^2\right\}.$$ 
The heat kernel and the Carnot--Carath\'eodory distance satisfy the following symmetry properties: 
\begin{gather*}
	p(O \, x, O \, t) = p(x,t), \quad p(x,t) = p(x, -t), \qquad \forall \, O \in \mathrm{O}_3, (x,t)\in\ntt,\\
d(O \, x, O \, t) = d(x,t), \quad d(x,t) = d(x, -t), \qquad \forall \, O \in \mathrm{O}_3, (x,t)\in\ntt,
\end{gather*}
where $\mathrm{O}_3$ represents the orthogonal group of order 3. Hence $p(x,t) = p(|x|e_1,(T_1,T_2,0))$, for any $(x,t)\in\ntt$, where
\begin{equation*}
	T_1=T_1(x,t): = |t|\,\hat{x}\cdot\hat{t}, \quad	T_2= T_2(x,t):=|t|\sqrt{1-(\hat{x}\cdot\hat{t})^2}.
\end{equation*}
From these facts and the sharp estimates of the heat kernel on $N_{3,2}$ (see \cite[Theorem 2.6]{LMZ23}), together with \cite[Lemma 11]{LZ21}, one can appeal to a simple limiting argument  and obtain the following
\begin{lem} \label{lem3}
For any $g=(x,t)\in\ss$, we put
\begin{equation*}
(\zeta,2\tz)=\exp^{-1}(g) \in \mathcal{D}, \quad	\ep=\ep(g):=\vtz_1-|\tz|>0, \quad \m=\m(g) := \frac14[d(g)^2 - |x|^2] > 0.
\end{equation*}
Then
	\begin{equation} \label{pbnd1}
		\begin{aligned}
			p(x,t) \sim (1+d(g))^{-2}\frac{1+\ep  \, d(g)}{1+\ep \,  d(g) + \ep  \, T_2^{\frac12} \, |x|^{\frac12} \,\m^{\frac14}} \,
			 e^{-\frac{d(g)^2}{4}}.
		\end{aligned}
	\end{equation}
\end{lem}

Now we claim that \eqref{assmcp} holds uniformly on $[0,1]$  with $N=N_0 +2$, where $N_0\ge 14$ is a constant such that the property $\mathrm{MCP}(0,N_0)$ holds on $\ntt$ by Theorem \ref{tmcp}.
Note that the smallest $N_0$ (i.e., the curvature exponent) must be larger than the geodesic dimension, which is $14$ for $\ntt$. See \cite[Proposition 7.3]{GZ24} for more details.

For our purpose we need the following fact concerning the small-time heat kernel asymptotics, which indeed provides a shortcut to derive the precise bounds for the Jacobian determinant of the sub-Riemannian exponential map from the precise estimates of the heat kernel. In fact, this result is hidden in \cite{BA88} and we are informed by the author. For the sake of completeness, we will include a detailed proof in Appendix \ref{app2}. To be more precise, we will also use the small-time heat kernel asymptotics in \cite[\S~4.2]{Li21} and \cite[Theorem 2.4]{Li21} for $g$ in a subset of $\mathcal{S}$ and a midpoint argument (see for example \cite{BBN12}) for general points in $\mathcal{S}$.
%More details of this approach will be included in the forthcoming paper.

\begin{lem} \label{sth}
Let $\G$ be a step-two Carnot group. For any $g = (x,t) \in \ss$ we have 
    $$
    p_h(x, t)=C(\G) \, h^{-\frac{q + m}{2}}\jac(\exp)(\zeta, \tau)^{-\frac{1}{2}} e^{-\frac{d(g)^2}{4 h}} \, (1+o_g(1)),\quad {\rm as}\,\, h \to 0^+,
    $$
   where $(\zeta, \tau) = \exp^{-1}(x,t) \in \mathcal{D}$ and  $C(\G)$ is a constant depending only on $\G$.
\end{lem}

Here and subsequently we will always suppose that $g=(x,t)=\exp(\zeta,2\tz) \in\ss$. Then from Lemma \ref{lem3} (together with \cite[(2.9)]{LZ21}),
\begin{align*}
	p_h(x,t) &= h^{-\frac92} p(h^{-\frac12}|x|e_1, h^{-1}(T_1,T_2,0)) \\ &\sim  h^{-\frac92} (1+h^{-\frac12}d) ^{-2} \frac{1+h^{-\frac12} \ep d} {1+ h^{-\frac12} \ep d  + h^{-1} \ep  \, T_2^{\frac12} \, |x|^{\frac12} \,\m^{\frac14}} e^{-\frac{d^2}{4h}}.
\end{align*}
Now combining this with Lemma \ref{sth}, it implies that
\[
C(\G) \, \jac(\exp)(\zeta, 2\theta)^{-\frac{1}{2}} (1 + o_g(1)) = p_h(x,t) e^{\frac{d(x,t)^2}{4h}} h^3 \sim (h^\frac12 + d)^{-2} \frac{h^\frac12 + \ep d}{h + h^\frac12 \ep d + \ep  T_2^{\frac12} \, |x|^{\frac12}\,\m^{\frac14}}.
\]
Letting $h \to 0^+$ we have
\begin{equation}\label{jac_exp}
	\jac(\exp)(\zeta,2\theta) \sim d(g)^2\, T_2(g) \,|x|\,\m(g)^\frac12 =: d^2 \,\aa, \quad\forall\, g =\exp(\zeta,2\tz) \in\ss,
\end{equation}
where $\aa=\aa(g):=T_2(g) \,|x|\,\m(g)^\frac12$. For brevity we also write for any $s\in(0,1]$,
\begin{gather*}
g_s:=\exp(s(\zeta,2\tz)), \quad d_s := d(g_s)=s\,d, \quad \aa_s :=\aa(g_s), \quad\ep_s := \ep(g_s) =\vtz_1 - s|\tz|.
\end{gather*}

Now by \eqref{pbnd1} and \eqref{jac_exp}, together with Proposition \ref{HCP} (ii), for our goal it suffices to prove that 
\begin{equation} \label{ine1}
    \frac{H(\zeta, 2 \theta) \, \jac(\exp)(\zeta, 2 \theta)}{H(s(\zeta, 2 \theta)) \, \jac(\exp)(s(\zeta, 2 \theta))} \sim \frac{(1+\ep d)(1+\ep_s d_s +\ep_s\aa_s^\frac12)(1+d_s^2)\aa \, d^2}{(1+\ep_s d_s)(1+\ep d +\ep\,\aa^\frac12)(1+d^2)\aa_s \, d_s^2} \lsim s^{4-N_0}
\end{equation}
holds for all $g =\exp(\zeta,2\tz) \in\ss$. Since $d_s=s\,d\le d$, then \eqref{ine1} will hold provided that 
\begin{equation} \label{ine2}
   (1+\ep d)(1+\ep_s d_s +\ep_s\aa_s^\frac12)\aa \lsim (1+\ep_s d_s)(1+\ep d +\ep\,\aa^\frac12)\widetilde{\aa_s},
\end{equation}
where $\widetilde{\aa_s}:=s^{6-N_0} \,\aa_s$. 

Notice that by $\mathrm{MCP}(0,N_0)$, Proposition \ref{HCP} (i) and \eqref{jac_exp} we have
$s^{N_0-6} d^2 \,\aa\lsim s^2 d^2\,\aa_s$, that is, $\aa\lsim s^{8-N_0}\aa_s$. In conclusion, we obtain
\begin{equation} \label{ine3}
    \aa\lsim s^2\,\widetilde{\aa_s},\qquad \aa_s\le \widetilde{\aa_s}.
\end{equation}
Using this and the trivial fact $\ep\le\ep_s$, it is easily seen that
\[
\left( \frac{1}{\ep_s} + d_s + \aa_s^\frac12 \right) \aa \lesssim s \left( \frac{1}{\ep} + d + \aa^\frac12 \right) \widetilde{\aa_s},
\]
which implies
\begin{equation*}
    \ep d(1+\ep_s d_s +\ep_s\aa_s^\frac12)\aa \lsim \ep_s d_s(1+\ep d +\ep\,\aa^\frac12)\widetilde{\aa_s}.
\end{equation*}
As a result, to get \eqref{ine2} we only need to show that
\begin{equation*}
    (1+\ep_s d_s +\ep_s\aa_s^\frac12)\aa \lsim (1+\ep_s d_s)(1+\ep d +\ep\,\aa^\frac12)\widetilde{\aa_s}.
\end{equation*}
But this follows from the estimates (by \eqref{ine3})
\begin{gather}
(1+\ep_s d_s) \aa \lsim (1+\ep_s d_s) s^2\, \widetilde{\aa_s} \le (1+\ep_s d_s)  \widetilde{\aa_s},\nonumber\\
\ep_s \aa_s^\frac12 \aa \le \ep_s \widetilde{\aa_s}^\frac12  \aa^\frac12 \aa^\frac12 \lesssim \ep_s \widetilde{\aa_s}  s\, \ep\, d^2 = \ep_s \, d_s \,\ep\, d \, \widetilde{\aa_s},\label{ine4}
\end{gather}
where in ``$\lsim$" of \eqref{ine4} we have also used that $\aa\lsim\ep^2d^4$, which, in fact, is a simple consequence of \cite[Lemma 7.6 (ii)]{LMZ23}.

\begin{appendices}
\section{Proof of Lemma \ref{geh}} \label{app}

Finally, we explain how to prove Lemma \ref{geh}, which is in fact contained in  \cite[Theorem 3.1]{Q15} if one can get rid of the finite measure assumption. For the reader's convenience, throughout this appendix we will work in the framework of \cite[p. 609]{Q15} with the same notations therein.

Notice that the generalized curvature-dimension inequality $\mathrm{CD}(\rho_1,\rho_2,k,d)$ for $L$ satisfying $\rho_1>0$ will imply that the invariant measure is finite (and even more that the underlying manifold $M$ must be compact; see \cite[Theorem 10.1]{BG17}). This obviously will not happen in our situation of step-two Carnot groups. However, one can show that \cite[Theorem 3.1]{Q15} still holds without the finite measure assumption as the following theorem says.

\begin{theo} \label{fint}
Suppose that the operator $L$ obeys the   generalized curvature-dimension inequality $\mathrm{CD}(\rho_1,\rho_2,k,d)$ with $\rho_1\ge0,\,\rho_2>0,\,k\ge0,\,d<\infty$. Then there is a positive constant $C=C(\rho_2,k,d)$ such that 
\begin{equation} \label{ff01}
    \Gamma_y(\ln p(x, y, t))+\rho_2\, t\, \Gamma_y^Z(\ln p(x, y, t)) \leq \frac{C}{t}\left(1+\frac{d(x, y)^2}{t}\right), \quad \forall\, t>0, x, y \in M.
\end{equation}
\end{theo}

\begin{proof}
The proof is similar to  \cite[Theorem 3.1]{Q15} with minor modifications. Recall that we still have the doubling volume property and two-sided bounds for the heat kernel as formulated in \cite[\S~3]{Q15}.

Fix $x\in M$ and $t>0$. Let $f:=p(x,\cdot,t)\in C_b^\infty(M)$. By the upper bound in \cite[Eq. (3.1)]{Q15} we obtain
\begin{equation*}
    0<f(y)\le \frac{C_2(\varepsilon)}{\mu(B(x,\sqrt{t}))} =: A(x,t)=A, \quad \forall\, y\in M,
\end{equation*}
where the choice $\varepsilon=1/2$ is enough for our purpose. 
Moreover, from the semigroup property of $P_t=e^{tL}$ and the upper bound in \cite[Eq. (3.1)]{Q15} it follows that
\begin{align*}
    P_t(f)(y)&=\int_M p(y,z,t)\,p(x,z,t) \dif\mu(z) \\
    &=p(x,y,2t) \ge \frac{C_2^{-1}(\varepsilon)}{\mu(B(x, \sqrt{2t}))} \exp \left(-\frac{\left(1+\frac{3 k}{2\rho_2}\right) d(x, y)^2}{2 (4-\varepsilon) t}\right).
\end{align*}

As a result,  substituting $f$ into  \cite[Eq. (2.2)]{Q15} (together with the fact that $\mathrm{CD}(\rho_1,\rho_2,k,d) \Rightarrow \mathrm{CD}(\rho_1,\rho_2,k,\infty)$ by definition) yields that
\begin{align*}
    &\frac{\rho_2 \,t}{\rho_2+2 k} \Gamma_y\left(\ln p(x,y,2t) \right) +\frac{\rho_2^2 \, t^2}{\rho_2+2 k} \Gamma^Z_y\left(\ln p(x,y,2t)\right) \leq \ln \frac{A}{P_t f(y)} \\
    &\qquad \le \ln\left[\frac{C_2^{2}(\varepsilon) \mu(B(x, \sqrt{2t}))}{\mu(B(x, \sqrt{t}))} \exp \left(\frac{\left(1+\frac{3 k}{2\rho_2}\right) d(x, y)^2}{2 (4-\varepsilon) t}\right) \right] \\
    &\qquad\le \ln(2^{Q/2}C_1 C_2^{2}(\varepsilon)) + \frac{1+\frac{3 k}{2\rho_2} }{ 4-\varepsilon} \frac{d(x, y)^2}{2t},
\end{align*}
where in the last ``$\le$" we have used the  doubling property (the inequality above \cite[Eq. (3.1)]{Q15}). In view this we get \eqref{ff01} since $t$ is arbitrary.
\end{proof}

Now Lemma \ref{geh} follows immediately as we have mentioned earlier.

\section{Proof of Lemma \ref{sth}} \label{app2}

In this appendix we prove Lemma \ref{sth}. Let us first recall what has been obtained in \cite{BA88} as  follows:
\begin{theo}[Theorem 3.1 of \cite{BA88}]\label{TBA88}
Let $\G$ be a step-two Carnot group. Given any compact set $K \subset \mathcal{S}$,  we have
\[
p_h(x,t) = \mathfrak{p}(x,t) h^{-\frac{q + m}{2}} \, e^{-\frac{d(g)^2}{4h}} \, (1 + O_K(h))
\]
as $h \to 0^+$ uniformly for  all $g \in K$. Here $\mathfrak{p} > 0$ is some smooth function
defined on $\mathcal{S}$.
\end{theo}

Then we recall the results in \cite{BBN12}, in which a midpoint argument is used. The basic idea of the midpoint argument is that we can re-write the heat kernel as the convolution of two heat kernels:
\[
p_h(g) = \int_\G p_{\frac{h}{2}}(g_*) p_{\frac{h}{2}}(g_*^{-1} \cdot g)  \mathrm{d}g_*.
\]
It turns out the the main contribution of this integral is near the midpoint of $o$ and $g$, and thus we can use the expansion in Theorem \ref{TBA88} to each heat kernel inside and ``glue together'' two pieces again by the integral.  We collect the results we need in the following proposition.

%Next we consider the  at $g\in\G$ given by

%Our key %proposition 
%tool is the following:

\begin{prop}\label{keyPro}
Suppose that $\G$ is a step-two Carnot group $\G$ and $g \in \mathcal{S}$.  
\begin{enumerate}[(i)]
\item (Lemma 21 of \cite{BBN12}) The hinged energy function
\begin{align*} %\label{Hin}
G_g(g_*) := \frac{1}{2}(d(g_*)^2 + d(g_*^{-1} \cdot g)^2),\qquad \forall \, g_* \in \G
\end{align*}
attains its minimum $\frac{d(g)^2}{4}$ exactly at the point $g_\mathrm{mid} :=  \exp \left( \frac{1}{2} \exp^{-1}(g) \right)$,  the midpoint of the shortest geodesic from $o$ to $g$;

\item If $\exp(\zeta,\tau)=g$, then we have $g_\mathrm{mid}, g_\mathrm{mid}^{-1} \cdot g \in \ss$ and furthermore
\begin{equation} \label{ff02}
    \exp^{-1}(g_\mathrm{mid}) = \left(\frac{\zeta}2,\frac{\tau}2\right), \qquad \exp^{-1}(g_\mathrm{mid}^{-1} \cdot g) = \left(\frac12 e^{\tilde{U}(\tau/2)}\zeta,\frac{\tau}2\right);
\end{equation}

\item (Theorem 24 of \cite{BBN12}) The Hessian matrix of $G_g$ at the point $g_\mathrm{mid}$ is nondegenerate and, moreover, by (i)-(ii) we have
\[ \mathrm{Hess}(G_g)(g_\mathrm{mid}) > 0; \]

\item (The first note after the proof of Theorem 27 on \cite[pp. 400]{BBN12}) It holds that
\end{enumerate} 
\begin{equation} \label{ff03}
 \mathfrak{p}(g) =  \mathfrak{p}(g_\mathrm{mid})  \mathfrak{p}(g_\mathrm{mid}^{-1} \cdot g) \frac{(8 \pi)^{\frac{q + m}{2}}}{\sqrt{\det(\mathrm{Hess}(G_g)(g_\mathrm{mid}))}}.
\end{equation}
%where $g_\mathrm{mid} :=  \exp \left( \frac{1}{2} \exp^{-1}(g) \right)$,  the midpoint of the shortest geodesic from $o$ to $g$.
\end{prop}

\begin{proof}
We only need to prove (ii). In fact, the fist equality in \eqref{ff02} follows from the definition of the midpoint and Proposition \ref{propgeo}. For the second one, noting  that we have $g_*^{-1}\in\ss$ if $g_* \in \ss$ by \eqref{lid2}, then from \eqref{symExp} we see that it suffices to prove $g^{-1} \cdot g_\mathrm{mid} = \exp\left(\frac12\left(- e^{\widetilde{U}(\tau)} \, \zeta, - \tau\right)\right) \in\ss$. In fact, since $\gamma(s) := \exp\{s(\zeta,\tau)\}$ is the shortest geodesic joining $o$ to $g$, then the curve $g^{-1} \cdot \gamma(1-s)$ is the shortest geodesic joining $o$ to $g^{-1}$, which, however, should also be $\exp\{s(- e^{\widetilde{U}(\tau)} \, \zeta, - \tau)\}$ by \eqref{symExp}. Since $g^{-1} \in \mathcal{S}$, the two shortest geodesics must coincide. From this with the choice that $s = 1/2$,  and Proposition \ref{propgeo}, the desired result follows.

%As for \eqref{ff03}, it is a direct consequence of the conclusion in the second paragraph on , up to some notational modification.
\end{proof}

\subsection{Proof of Lemma \ref{sth}}

Following the notation in \cite{Li21}, we define $U(\tau) := i \widetilde{U}(\tau)$ and $\Omega_* := \{\tau; \, \|U(\tau)\| < \pi\}$, where $\|\cdot\|$ is the operator norm. It follows from \cite[Theorem 2.4]{Li21} that if $\tau \in 2\Omega_*$, then the sub-Riemannian exponential map is a composition of $\mathcal{E}_1$ and $\mathcal{E}_2$, that is, $\exp(\zeta,  \tau) = \mathcal{E}_2 \circ \mathcal{E}_1(\zeta,  \tau)$, where
\begin{align*}
\mathcal{E}_1: \R^q \times 2\Omega_* &\longrightarrow \R^q \times \Omega_* \\
(\zeta, \tau) &\longmapsto \left( \frac{\sin{U(\tau/2)}}{U(\tau/2)} e^{\widetilde{U}(\tau/2)} \, \zeta  , \tau/2 \right),
\end{align*}
and 
\begin{align*}
\mathcal{E}_2: \R^q \times \Omega_* & \longrightarrow \R^q \times \R^m \\
(x, \tau) & \longmapsto \left( x, -\frac{1}{4} \, \nabla_\tau \, \langle U(\tau) \cot{U(\tau)} \, x, \, x \rangle \right).
\end{align*}
Here $\nabla_\tau$ is the usual Euclidean gradient w.r.t. $\tau$ variable.  Taking Jacobian determinant on $\exp = \mathcal{E}_2 \circ \mathcal{E}_1$, we can re-write the small-time heat kernel asymptotics in \cite[\S~4.2]{Li21} as 
$$
    p_h(x, t)=C(\G) \, h^{-\frac{q + m}{2}}\jac(\exp)(\zeta, \tau)^{-\frac{1}{2}} e^{-\frac{d(g)^2}{4 h}} \, (1+o_g(1)),\quad {\rm as}\,\, h \to 0^+,
    $$
for some positive constant $C(\G) > 0$ and $(\zeta, \tau) = \exp^{-1}(x,t) \in \mathcal{D} \cap (\R^q \times 2\Omega_*)$. In other words, we know that for $g \in \exp(\mathcal{D} \cap (\R^q \times 2\Omega_*))$ ,
\[
\mathfrak{p}(g) = C(\G) \jac(\exp)(\exp^{-1}(g))^{-\frac{1}{2}},
\]
which is a real analytic function since the functions $d^2$ and $\exp^{-1}$ are both real analytic on $\ss$ (see Proposition \ref{propgeo} and \eqref{Expexp}). As a result, the functions $g \mapsto g_\mathrm{mid},\, g_\mathrm{mid}^{-1} \cdot g \mbox{ and } \det(\mathrm{Hess}(G_g)(g_\mathrm{mid}))(>0))$ are also real analytic on $\ss$.

Now for $g \in \exp(\mathcal{D} \cap (\R^q \times 4\Omega_*))$, by (ii) and (iv) of Proposition \ref{keyPro} we see that $\mathfrak{p}$ is also a real analytic function on $\exp(\mathcal{D} \cap (\R^q \times 4\Omega_*))$. So by the  uniqueness theorem for the real analytic functions (choosing connected component if necessary), we obtain  
\[
\mathfrak{p}(g) = C(\G) \jac(\exp)(\exp^{-1}(g))^{-\frac{1}{2}}
\]
on $\exp(\mathcal{D} \cap (\R^q \times 4\Omega_*))$. Hence Lemma \ref{sth} can be proved by repeating this argument.

\end{appendices}

\section*{Acknowledgements}
\setcounter{equation}{0}
The authors would like to thank Prof. G\'erard Ben Arous for informing us of Lemma \ref{sth}. They would like to thank Prof. Hong-Quan Li and the anonymous referees for their many useful suggestions and valuable remarks which improve the writing of the paper. The work of YZ was supported by JSPS KAKENHI Grant Number JP24K16928.

%\nocite{*}
%\phantomsection\addcontentsline{toc}{section}{References}
%\bibliographystyle{abbrv}
%\bibliography{LMZnhk}

\mbox{}\\
Sheng-Chen Mao \\ 
School of Mathematical Sciences 
Fudan University \\
220 Handan Road  \\
Shanghai 200433  \\
People's Republic of China \\
E-Mail: scmao22@m.fudan.edu.cn \quad or \quad maosci@163.com \\

\mbox{}\\
Ye Zhang\\
Analysis on Metric Spaces Unit  \\
Okinawa Institute of Science and Technology Graduate University \\
1919-1 Tancha, Onna-son, Kunigami-gun \\
Okinawa, 904-0495, Japan \\
E-Mail: zhangye0217@gmail.com \quad or \quad Ye.Zhang2@oist.jp \mbox{}\\

\end{document}